\documentclass{cmslatex}
\usepackage[paperwidth=7in, paperheight=10in, margin=.875in]{geometry}
 \usepackage[backref,colorlinks,linkcolor=red,anchorcolor=green,citecolor=blue]{hyperref}
\usepackage{amsfonts,amssymb}
\usepackage{amsmath}
\usepackage{graphicx}
\usepackage{enumerate}
\renewcommand{\theequation}{\arabic{section}.\arabic{equation}}

\allowdisplaybreaks

\usepackage{commath}
\usepackage{mathtools}

\usepackage{enumitem}

\usepackage{caption}
\usepackage{subcaption}

\newcommand*{\R}{\mathbb{R}}

\usepackage{dsfont}
\newcommand{\avg}[1]{\mathds{E}\!\left[#1\right]}
\newcommand{\var}[1]{\mathds{V}\!\left[#1\right]}

\newcommand{\T}{^{\!\!\intercal}}

\newcommand{\Gauss}[1]{\mathcal{N}\!\left(#1\right)}
\newcommand{\isGauss}[1]{\sim \Gauss{#1}}

\let\oldepsilon\epsilon
\renewcommand\epsilon\varepsilon

\newcommand{\setc}[2]{\left\{#1\,\middle|\,#2\right\}}

\RequirePackage[disable,colorinlistoftodos,prependcaption,textsize=scriptsize]{todonotes}
\RequirePackage{xparse, xargs}
\newcommandx{\note}[2][1=]{
	\todo[inline,linecolor=black,backgroundcolor=violet!5,bordercolor=violet,#1]{
		NOTE: #2
	}
}
\newcommandx{\td}[2][1=]{
	\todo[inline,linecolor=orange,backgroundcolor=orange!5,bordercolor=orange,tickmarkheight=2cm,#1]{
		TODO: #2
	}
}
\newcommandx{\lb}[2][1=]{
\todo[linecolor=teal,backgroundcolor=teal!5,bordercolor=teal,#1]{
		#2
	}
}

\newcommandx{\jm}[2][1=]{
\todo[linecolor=violet,backgroundcolor=violet!5,bordercolor=violet,#1]{
		#2}}

\newcommandx{\sab}[2][1=]{
\todo[linecolor=green,backgroundcolor=green!5,bordercolor=green,#1]{
		#2
	}
}

\newcommand{\numberthis}{\addtocounter{equation}{1}\tag{\theequation}}

\newtheorem{hypo}{Hypothesis}[section]

\usepackage[capitalise,nameinlink,noabbrev]{cleveref}
\crefname{hypo}{Hypothesis}{Hypotheses}
\crefname{equation}{}{}

\usepackage[square,numbers,sort]{natbib}

\newcommand{\email}[1]{\protect\href{mailto:#1}{#1}}

 \title{The convergence of stochastic differential equations to their linearisation in small noise limits\thanks{Received date, and accepted date. \\
 LB acknowledges support from an Australian Government Research Training Program Scholarship. SB acknowledges with thanks partial support from the Australian Research Council via grant DP200101764.}}

\author{Liam Blake\thanks{School of Computer and Mathematical Sciences, University of Adelaide, Adelaide 5005, Australia (corresponding author: \email{liam.blake@adelaide.edu.au}).} \and John Maclean\thanks{School of Computer and Mathematical Sciences, University of Adelaide, Adelaide 5005, Australia.} \and Sanjeeva Balasuriya\footnotemark[3]}

\begin{document}
\pagestyle{myheadings}
\markboth{CONVERGENCE OF SDES TO THEIR LINEARISATIONS}{BLAKE, MACLEAN \& BALASURIYA}

\maketitle

\begin{abstract}
	Prediction via deterministic continuous-time models will always be subject to model error, for example due to unexplainable phenomena, uncertainties in any data driving the model, or discretisation/resolution issues. In this paper, we build upon previous small-noise studies to provide an explicit bound for the error between a general class of stochastic differential equations and corresponding computable linearisations written in terms of a deterministic system. Our framework accounts for non-autonomous coefficients, multiplicative noise, and uncertain initial conditions. We demonstrate the predictive power of our bound on diverse numerical case studies. We confirm that our bound is \emph{sharp}, in that it accurately predicts the error scaling in the moments of the linearised approximation as both the uncertainty in the initial condition and the magnitude of the noise in the differential equation are altered. This paper also provides an extension of stochastic sensitivity, a recently introduced tool for quantifying uncertainty in dynamical systems, to arbitrary dimensions and establishes the link to our characterisation of stochastic differential equation linearisations.
\end{abstract}



\section{Introduction}
Stochastic differential equations (SDEs) are a natural framework for introducing different sources of uncertainty 
into the continuous time evolution of a variable \cite{Oksendal_2003_StochasticDifferentialEquations,SarkkaSolin_2019_AppliedStochasticDifferential,KallianpurSundar_2014_StochasticAnalysisDiffusion}.
Generally, in modelling situations the dynamics are highly nonlinear and one expects the noise to be multiplicative (i.e. vary with state), e.g. in atmospheric \cite{SuraEtAl_2005_MultiplicativeNoiseNonGaussianity} and oceanic \cite{KamenkovichEtAl_2015_PropertiesOriginsAnisotropic} systems and from experimental and observational considerations.
Such SDEs are usually intractable to solve analytically and computationally expensive to approximate accurately.
Data-based models---that is, models possessing terms in the equations which are driven by data rather than by explicitly specified functions---and uncertainty in the initial state render additional problems in obtaining a theoretical understanding of the stochastic system.

A common approach to characterising and approximating the solution of nonlinear SDEs with small noise is via a ``linearisation'' through time about a single deterministic trajectory.
A linearised stochastic differential equation, obtained by truncating Taylor expansions of each coefficient \cite[e.g.]{Jazwinski_2014_StochasticProcessesFiltering,Blagoveshchenskii_1962_DiffusionProcessesDepending}, can be solved analytically and is accordingly used across a diversity of literature and applications \cite{Jazwinski_2014_StochasticProcessesFiltering,SarkkaSolin_2019_AppliedStochasticDifferential,KaszasHaller_2020_UniversalUpperEstimate,ArchambeauEtAl_2007_GaussianProcessApproximations,Sanz-AlonsoStuart_2017_GaussianApproximationsSmall,LawEtAl_2015_DataAssimilationMathematical,ReichCotter_2015_ProbabilisticForecastingBayesian,BudhirajaEtAl_2019_AssimilatingDataModels}.

Much is already known about these linearisations; classical results in the context of small-noise series expansions \cite{Blagoveshchenskii_1962_DiffusionProcessesDepending} and large deviations theory \cite{FreidlinWentzell_1998_RandomPerturbationsDynamical} show that the strong error between SDE solution with a fixed initial condition and that of an appropriate linearisation is bounded.
\citet{Sanz-AlonsoStuart_2017_GaussianApproximationsSmall} establish a strong result, bounding the Kullback-Leibler divergence between the solutions of autonomous SDEs with additive stationary noise and a linearised equivalent. Their result considers both an uncertain initial condition, and the evolving error due to the discrepancy between the models.

In \cref{sec:theory}, we directly bound the expectation of all moments of the distance between the linearisation and the true solution of the stochastic differential equation.
Our result extends the former convergence result of \cite{Sanz-AlonsoStuart_2017_GaussianApproximationsSmall} in three ways: we predict the scaling of all moments of the error, in the presence of multiplicative noise terms, and including time-dependent coefficients in the stochastic differential equations. We directly compare our results with those of \cite{Sanz-AlonsoStuart_2017_GaussianApproximationsSmall} in \cref{sec:comparison}.
In \Cref{sec:numerics}, we show numerically that the scaling of the strong error predicted by \Cref{thm:main} matches our predictions on three example SDEs in \(1\)- and \(2\)-dimensions.

The second contribution of the paper is to the notion of ``stochastic sensitivity'' \cite{Balasuriya_2020_StochasticSensitivityComputable}, which seeks to explicitly quantify the impact of vector field uncertainty on the solution trajectories of dynamical systems \cite{BranickiUda_2021_LagrangianUncertaintyQuantification,KaszasHaller_2020_UniversalUpperEstimate,BranickiUda_2023_PathBasedDivergenceRates,Balibrea-IniestaEtAl_2016_LagrangianDescriptorsStochastic}.
This methodology was originally developed by Balasuriya \cite{Balasuriya_2020_StochasticSensitivityComputable} as a tool for determining Lagrangian coherent structures (LCS) \cite{BalasuriyaEtAl_2018_GeneralizedLagrangianCoherent,HadjighasemEtAl_2017_CriticalComparisonLagrangian} in fluid flows, in that clusters of trajectories which have small uncertainty may be thought of as more ``coherent'' than others.
In \Cref{sec:theory_s2}, we generalise the (formerly two-dimensional) stochastic sensitivity to arbitrary dimensions by proving that the stochastic sensitivity can be computed as the largest eigenvalue of the covariance matrix of the solution to the linearised equation from \cref{sec:theory}.
Our expressions extend \cite{Balasuriya_2020_StochasticSensitivityComputable} by empowering a rapid computation of stochastic sensitivity, as previously the computation required three integrals of particular rotations of the gradient of the flow map.
We demonstrate this computation in \cref{sec:comput_s2}.

\section{Convergence of a SDE to a linearisation}\label{sec:theory}
Suppose we are interested in the evolution of a \(\R^n\)-valued state variable \(y_t\) over a finite time interval \([0,T]\).
Our model, accounting for uncertainties arising from a range of sources, for the evolution of this variable is the It\^o stochastic differential equation
\begin{equation}
	\dif y_t^{(\epsilon)} = u\!\left(y_t^{(\epsilon)}, t\right)\dif t + \epsilon \, \sigma\!\left(y_t^{(\epsilon)}, t\right)\dif W_t, \quad y_0^{(\epsilon)} = x
	\label{eqn:sde_y}
\end{equation}
where \(u\colon \R^n \times [0,T] \to \R^n\) is the governing reference vector field. 
The canonical \(m\)-dimensional Wiener process \(W_t\)  is a continuous white-noise stochastic process with independent Gaussian increments.
The scale of the ongoing noise is assumed to be small and is parameterised as \(0 < \epsilon \ll 1\).
The noise in \cref{eqn:sde_y} is multiplicative, in that the diffusion matrix \(\sigma\colon \R^n \times [0,T] \to \R^{n\times m}\) can vary with state \( x \), as well as with time \( t \).
We assume that \(\sigma\) is specified \textit{a priori}, or if no such information is known, then \(\sigma \equiv I\), the \(n \times m\) identity matrix, is a default modelling choice.
We consider \cref{eqn:sde_y} subject to the \emph{general} uncertain initial condition \(y_0^{(\epsilon)} = x\), where \(x\) is an \(n\)-dimensional random vector with some given distribution. The two sources of randomness, $ x $ and $ W_t $, are assumed independent from each other.

In the absence of any uncertainty (i.e. \(\epsilon = 0\) and the initial condition is a known deterministic quantity), \cref{eqn:sde_y} reduces to the ordinary differential equation
\begin{equation}
	\dod{y_t^{(0)}}{t} = u\!\left(y_t^{(0)}, t\right), \quad y_0^{(0)} = x_0.
	\label{eqn:ode_det}
\end{equation}
where the initial condition \(x_0 \in \R^n\) is fixed.
The formal convergence of the stochastic solution \(y_t^{(\epsilon)}\) (under certain conditions on the initial condition) to the deterministic \(y_{t}^{(0)}\) in the limit as \(\epsilon \to 0\) is well-established using the large deviations principle \cite[e.g]{FreidlinWentzell_1998_RandomPerturbationsDynamical}.
We refer to \cref{eqn:ode_det} as the \emph{reference} deterministic model associated with \cref{eqn:sde_y}.
Solutions to the reference deterministic model are more readily available, e.g. in terms of computational efficiency when solving numerically, than those of the stochastic model, but do not account for inevitable uncertainty.

Let the flow map \(F_{0}^{t}\colon \R^n \to \R^n\) be the function which evolves an initial condition from time \(0\) to time \(t\) according to the flow of \cref{eqn:ode_det}, i.e. \(F_0^t\!\left(x_0\right) = y_t^{(0)}\).

We assume certain smoothness and boundedness conditions on the various terms outlined, which are stated explicitly in \Cref{hyp:smooth}.
Throughout this article, we use the norm symbol \(\norm{\cdot}\) to denote (i) for a vector, the standard Euclidean vector norm, (ii) for a matrix, the spectral norm induced by the Euclidean norm, and (iii) for a 3rd-order tensor, the spectral norm induced by the matrix norm.
The gradient symbol \(\nabla\) generically refers to derivatives with respect to the state variable.

\renewcommand\thehypo{H}
\begin{hypo}\label{hyp:smooth}
	Let the deterministic functions \(u \colon \R^n\times [0,T] \to \R^n\) and \(\sigma \colon \R^n \times [0,T] \to \R^{n\times m}\), and the random initial condition \(x\) be such that:
	\begin{enumerate}[label=(H.\arabic{*}), ref=H.\arabic{*}]
		\item\label{hyp:fm_exists} For all \(t \in [0,T]\) flow map \(F_0^t \colon \R^n \to \R^n\) is well-defined, and continuously differentiable (with respect to the initial condition) with invertible derivative.

		\item\label{hyp:coef_cont} For each \(t \in [0,T]\), the function \(u(\cdot, t)\colon \R^n \to \R^n\) given by \(u(x,t)\) is twice continuously differentiable on \(\R^n\), and each component of the function \(\sigma(\cdot, t)\colon \R^n \to \R^{n\times m}\) given by \(\sigma(x,t)\) is differentiable on \(\R^n\).

		\item\label{hyp:u_bounds} There exists a constant \(K_{\nabla u} \geq 0\) such that for any \(t \in [0,T]\) and \(x \in \R^n\),
		\begin{equation*}
			\norm{\nabla u(x,t)} \leq K_{\nabla u}.
		\end{equation*}
		Equivalently, for all \(t \in [0,T]\), the function \(u\!\left(\cdot, t\right)\) is Lipschitz continuous with Lipschitz constant \(K_{\nabla u}\).

		\item\label{hyp:coef_meas} For each \(x \in \R^n\), the function \(u(x,\cdot) \colon [0,T] \to \R^n\) and each component of the function \(\sigma(x,\cdot) \colon [0,T] \to \R^{n\times m}\) are Borel-measurable on \([0,T]\).

		\item\label{hyp:linear_growth} There exists a constant \(K_L\) such that for any \(t \in [0,T]\) and \(x \in \R^n\),
		\[
			\norm{u\left(x,t\right)} + \norm{\sigma\left(x,t\right)} \leq K_L\left(1 + \norm{x}\right).
		\]

		\item\label{hyp:sigma_deriv_bound} There exists a constant \(K_{\nabla\sigma} \geq 0\) such that for any \(t \in [0,T]\) and \(x \in \R^n\),
		\begin{equation*}
			\norm{\nabla\sigma(x,t)} \leq K_{\nabla\sigma},
		\end{equation*}
		and we take \(K_{\nabla\sigma} = 0\) if there is no spatial dependence in \(\sigma\).
		Equivalently, for all \(t \in [0,T]\), the function \(\sigma\!\left(x, \cdot\right)\) is Lipschitz continuous with Lipschitz constant \(K_{\nabla\sigma}\).

		\item\label{hyp:init_indep} The initial condition \(x\) is defined on the same probability space as \(W_t\), and is independent of \(W_t\) for all \(t \in [0,T]\).

		\item\label{hyp:nnu_bounds} There exists a constant \(K_{\nabla\nabla u} \geq 0\) such that for any \(t \in [0,T]\) and \(x \in \R^n\),
		\[
			\norm{\nabla \nabla u(x,t)} \leq K_{\nabla\nabla u},
		\]
		and we take \(K_{\nabla\nabla} = 0\) if the second spatial derivatives of \(u\) are all zero.

		\item\label{hyp:sigma_bounds} There exists a constant \(K_\sigma \geq 0\) such that for any \(t \in [0,T]\) and \(x \in \R^n\),
		\begin{equation*}
			\norm{\sigma(x,t)} \leq K_{\sigma}.
		\end{equation*}

	\end{enumerate}
\end{hypo}
The conditions \ref{hyp:coef_cont} to \ref{hyp:init_indep} guarantee that \cref{eqn:sde_y} with the initial condition \(y_0 = x\) has a unique strong solution \cite{KallianpurSundar_2014_StochasticAnalysisDiffusion}.
The bound \(K_{\nabla\nabla u}\) placed on the second derivatives of \(u\) in \ref{hyp:nnu_bounds} quantifies exactly when the deterministic dynamics (that is, \(u\)) of \cref{eqn:sde_y} are linear.
Similarly, the bound \(K_{\nabla\sigma}\) on the spatial derivatives of \(\sigma\) in \ref{hyp:sigma_deriv_bound} allows us to distinguish when the noise in \cref{eqn:sde_y} is multiplicative.

Our aim is to construct and formally justify a computable linearisation of \cref{eqn:sde_y} about a trajectory solving the deterministic system \cref{eqn:ode_det}.
To that end, we take a \emph{fixed} initial condition \(x_0 \in \R^n\) to the reference deterministic model \cref{eqn:ode_det} and consider linearising the SDE \cref{eqn:sde_y} about the corresponding trajectory \(F_0^t\!\left(x_0\right)\).
We consider the following linearisation of \cref{eqn:sde_y}:
\begin{equation}
	\dif l_t^{(\epsilon)} = \left[u\!\left(F_0^t\!\left(x_0\right), t\right) + \nabla u\!\left(F_0^t\!\left(x_0\right), t\right)\left(l_t^{(\epsilon)} - F_0^t\!\left(x_0\right)\right)\right] \! \dif t + \epsilon\sigma\!\left(F_0^t\!\left(x_0\right), t\right) \! \dif W_t, \quad l_0^{(\epsilon)} = x,
	\label{eqn:linear_sde_inform}
\end{equation}
where the initial condition \(x \) is still permitted to be random.
Informally, we can arrive at \cref{eqn:linear_sde_inform} by performing a Taylor expansion of the coefficient \(u\) up to first-order and \(\sigma\) to zeroth-order about the time-varying trajectory \(F_0^t\!\left(x_0\right)\).
Such a linearisation is advantageous over the nonlinear SDE \cref{eqn:sde_y}, since \cref{eqn:linear_sde_inform} can be solved analytically.
We will later (see \Cref{cor:limit_moments}) provide explicit expressions for computing the distribution of the solution \(l_t^{(\epsilon)}\) solely in terms of the solution dynamics of the deterministic system \cref{eqn:ode_det}, the specified diffusion matrix \(\sigma\), and the distribution of \(x\).

In order to quantify the error arising from the choice of reference point \(x_0\), we define
\[
	\delta_r \coloneqq \avg{\norm{x - x_0}^r}^{1/r},
\]
i.e. \(\delta_r\) is the \(L_r\) distance between \(x\) and the deterministic point \(x_0\).
We can think of \(\delta_r\) as a scalar measure of the uncertainty in the initial condition, relative to the choice of reference point \(x_0\).
Alternatively, the limit as \(\delta_r\) approaches zero is equivalent to convergence in \(r\)th mean of \(x\) to the fixed point \(x_0\).
We can therefore distinguish two sources of uncertainty in our model; that arising from the initial condition, quantified by \(\delta_r\), and the ongoing uncertainty driven by the Wiener process \(W_t\) as measured by \(\epsilon\).

Our first and primary result, \Cref{thm:main}, provides an explicit bound on the \(r\)th moment of the error between the SDE solution \(y_t^{(\epsilon)}\) and the linearised solution \(l_t^{(\epsilon)}\).

\begin{theorem}[Linearisation error is bounded]\label{thm:main}
	Let \(y_t^{(\epsilon)}\) be the strong solution to the SDE \cref{eqn:sde_y} and \(l_t^{(\epsilon)}\) be the strong solution to the corresponding linearisation \cref{eqn:linear_sde_inform}, both driven by the same Wiener process \(W_t\) and subject to the same random initial condition \(y_0^{(\epsilon)} = l_0^{(\epsilon)} = x\).
	Then, for any \(r \geq 1\) such that \(\delta_{2r} < \infty\) and \(t \in [0,T]\), there exist constants \( D_1\!\left(r,t, K_{\nabla u}, K_{\sigma}\right), \, D_2\!\left(r,t, K_{\nabla u}\right), \, D_3\!\left(r,t, K_{\nabla u}\right) \in [0, \infty) \) independent of \(x\) and \(x_0\) such that for all \(\epsilon > 0\),
	\begin{equation}
		\avg{\norm{y_t^{(\epsilon)} - l_t^{(\epsilon)}}^r} \leq \begin{multlined}[t]
			\left(K_{\nabla\nabla u}^r + K_{\nabla\sigma}^r\right) D_1\!\left(r,t, K_{\nabla u}, K_\sigma\right)\, \epsilon^{2r} \\
			+ K_{\nabla\nabla u}^r D_2\!\left(r,t, K_{\nabla u}\right)\delta_{2r}^{2r}
			+ K_{\nabla\sigma}^r D_3\!\left(r,t, K_{\nabla u}\right)\delta_r^r \epsilon^r.
		\end{multlined}
		\label{eqn:main_ineq}
	\end{equation}
\end{theorem}

\begin{proof}
	See \Cref{app:main_thm_proof}.
	Our proof employs the Burkholder-Davis-Gundy inequality, Gr\"onwall's inequality, and Taylor's theorem to explicitly construct the bounding coefficients in terms of the conditions on the SDE coefficients set out in \Cref{hyp:smooth}.
	The bounding coefficients \(D_1\), \(D_2\), and \(D_3\) are given explicitly in \cref{eqn:bound_defns}.
\end{proof}

In \cref{eqn:main_ineq}, we have an explicit scaling of the error in terms of \(\epsilon\) and \(\delta_r\).
The three terms can be informally considered as: a contribution purely from the ongoing linearisation error, a contribution purely from the initial uncertainty, and a term resulting from the interaction between the initial and ongoing uncertainties.
By explicitly identifying the dependence of the bound on \(K_{\nabla\nabla u}\) and \(K_{\nabla \sigma}\), we note three special cases that are summarised by \Crefrange{rem:bound_linear}{rem:bound_exact}.

\begin{remark}[Linear drift]\label{rem:bound_linear}
	When the deterministic dynamics are linear, we set \(K_{\nabla\nabla u} = 0\) and \cref{eqn:main_ineq} becomes
	\[
		\avg{\norm{y_t^{(\epsilon)} - l_t^{(\epsilon)}}^r} \leq   K_{\nabla\sigma}^r D_1\!\left(r, t, K_{\nabla u}, K_\sigma\right)\, \epsilon^{2r} + K_{\nabla\sigma}^r D_3\!\left(r,t, K_{\nabla u}\right)\delta_r^r \epsilon^r.
	\]
	The linearisation of the drift term \(u\) is exact, so the error is purely due to the spatial dependency of the diffusion term \(\sigma\).
\end{remark}

\begin{remark}[Additive noise]\label{rem:bound_additive}
	When the noise in \cref{eqn:sde_y} is additive (non-multiplicative), we set \(K_{\nabla\sigma} = 0\) and \cref{eqn:main_ineq} becomes
	\[
		\avg{\norm{y_t^{(\epsilon)} - l_t^{(\epsilon)}}^r} \leq   K_{\nabla\nabla u}^r D_1\!\left(r,t, K_{\nabla u}, K_\sigma\right)\, \epsilon^{2r} + K_{\nabla\nabla u}^r D_2\!\left(r,t, K_{\nabla u}\right)\delta_{2r}^{2r}.
	\]
	The error is then purely due to the linearisation of the drift term \(u\), and as expected is of second order in both the initial condition uncertainty \(\delta_{2r}\) and the ongoing uncertainty \(\epsilon\).
\end{remark}

\begin{remark}[Exact linearisation]\label{rem:bound_exact}
	When the deterministic dynamics are linear and the noise in \cref{eqn:sde_y} is additive (non-multiplicative), the linearisation \cref{eqn:linear_sde_inform} should be exact.
	Accordingly, we set \(K_{\nabla\nabla u} = K_{\nabla\sigma} = 0\) and \cref{eqn:main_ineq} becomes
	\[
		\avg{\norm{y_t^{(\epsilon)} - l_t^{(\epsilon)}}^r} = 0.
	\]
	In turn, this implies that \(y_t^{(\epsilon)} = l_t^{(\epsilon)}\) almost surely, for any choice of reference point \(x_0\).
\end{remark}

We postpone a discussion of an additional special case---where the initial condition is fixed or
Gaussian---to a later section.  For the general situation,
we next explicitly establish the solution to the linearisation \cref{eqn:linear_sde_inform}, in terms of the initial condition and the deterministic evolution of \cref{eqn:ode_det}.

\begin{theorem}[Solution of the linearised SDE]\label{thm:limit_sol}
	The strong solution to the linearised SDE \cref{eqn:linear_sde_inform} is
	\begin{equation}
		l_t^{(\epsilon)} = \nabla F_0^t\!\left(x_0\right)\left(x - x_0\right) + F_0^t\!\left(x_0\right) + \epsilon\nabla F_0^t\!\left(x_0\right)\int_0^t{L\!\left(x_0, \tau\right)\dif W_\tau}.
		\label{eqn:linear_sol}
	\end{equation}
	where the term involving the uncertain initial condition \(x\) and the It\^o integral are independent, and
	\begin{equation}
		L\!\left(x_0, \tau\right) \coloneqq \left[\nabla F_0^\tau(x_0)\right]^{-1}\sigma\left(F_0^\tau(x_0), \tau\right).
		\label{eqn:sigma_L_def}
	\end{equation}

\end{theorem}
\begin{proof}
	See \Cref{app:limit_sol_proof}.
\end{proof}

The representation of the linearised solution as an independent sum in \cref{eqn:linear_sol} can be seen as a decomposition into contributions from the initial uncertainty (the transformation of initial condition \(x\)), a deterministic prediction (the flow map \(F_0^t\!\left(x_0\right)\)) and the ongoing uncertainty in \(u\) (the remaining It\^o integral term).

We can further show that the It\^o integral term follows a Gaussian random variable, which ensures that the independent sum in \cref{eqn:linear_sol} is a convenient expression for both theoretical analysis and numerical computation.
We also provide explicit expressions for the mean and covariance matrix of the linearised solution, written in terms of the deterministic dynamics and \(\sigma\).
\begin{corollary}[Distribution of the linearised solution]\label{cor:limit_moments}
	The It\^o integral term in \cref{eqn:linear_sol} follows a Gaussian distribution independent of \(x\), namely
	\[
		\int_0^t{L\!\left(x_0,\tau\right)\dif W_\tau} \isGauss{0, \int_0^t{L\!\left(x_0, \tau\right)L\!\left(x_0, \tau\right)^{\T}\dif\tau}}.
	\]
	The mean of the linearised solution is
	\begin{equation}
		\avg{l_t^{(\epsilon)}} = F_0^t\!\left(x_0\right) + \nabla F_0^t\!\left(x_0\right) \avg{x - x_0}.
		\label{eqn:mean_expl_eqn}
	\end{equation}
	The \(n\times n\) covariance matrix of the linearised solution is given explicitly by
	\begin{equation}
		\var{l_t^{(\epsilon)}} = \nabla F_0^t\!\left(x_0\right)\left(\var{x} + \epsilon^2 \int_0^t{L\!\left(x_0,\tau\right)L\!\left(x_0,\tau\right)^{\T}\dif\tau}\right)\left[\nabla F_0^t\!\left(x_0\right)\right]^{\T}
		\label{eqn:pi_expl_eqn}
	\end{equation}
	where \(L\!\left(x_0, \tau\right)\) is as defined in \cref{eqn:sigma_L_def} and the integral is taken in the elementwise sense.
\end{corollary}
\begin{proof}
	See \Cref{app:limit_moments_proof}.
	The expressions follow from the representation of the linearised solution as an independent sum in \cref{eqn:linear_sol}.
\end{proof}

In \Cref{thm:limit_sol}, we have provided expressions for the distribution of the solution \(l_t^{(\epsilon)}\) to the linearised SDE \cref{eqn:linear_sde_inform} written solely in terms of the behaviour of the deterministic system \cref{eqn:ode_det}, the specified diffusion matrix \(\sigma\), and the distribution of the initial condition \(x\).
This describes a method for approximating the solution to the nonlinear SDE \cref{eqn:sde_y}, or for characterising the impact of uncertainty in a dynamical system \cref{eqn:ode_det}, that circumvents the need for expensive stochastic simulation.

Thus far, we have stated our results in terms of a general initial condition \(x\), and provided expressions for the linearised solution in terms of this otherwise arbitrary distribution.
However, we later consider two special cases for the initial condition \(x\), a fixed (deterministic) initial condition in \Cref{sec:theory_fixed}, and a Gaussian initial condition in \Cref{sec:theory_gauss}.
In both these cases, the linearised solution also follows a Gaussian distribution which is characterised entirely by the mean and covariance described in \Cref{cor:limit_moments}, allowing for easy computation.
We also relate these results directly to other literature \cite{Jazwinski_2014_StochasticProcessesFiltering,FreidlinWentzell_1998_RandomPerturbationsDynamical,Blagoveshchenskii_1962_DiffusionProcessesDepending,Balasuriya_2020_StochasticSensitivityComputable,Sanz-AlonsoStuart_2017_GaussianApproximationsSmall,SarkkaSolin_2019_AppliedStochasticDifferential} which uses linearisation procedures and Gaussian process approximations for nonlinear SDEs in these situations.

Next, we establish the ordinary differential equation satisfied by the covariance matrix, which is an expression consistent with linearisations schemes described elsewhere \cite{ArchambeauEtAl_2007_GaussianProcessApproximations,SarkkaSolin_2019_AppliedStochasticDifferential,Jazwinski_2014_StochasticProcessesFiltering,Sanz-AlonsoStuart_2017_GaussianApproximationsSmall}.
This ODE enables rapid computation of the mean and covariance of the linearised solutions by solving a system of ODEs, i.e. \cref{eqn:ode_det} and \cref{eqn:pi_ode}.

\begin{remark}\label{rem:cov_ode}
	The \(n\times n\) covariance matrix \(\var{l_t^{(\epsilon)}}\) of the linearised solution is the symmetric positive-semidefinite \(n \times n\) matrix solution to the ordinary differential equation
	\begin{equation}
		\dod{\Pi(t)}{t} = \begin{multlined}[t]
			\nabla u\!\left(F_0^t\!\left(x_0\right), t\right) \Pi(t) + \Pi(t)\left[\nabla u\!\left(F_0^t\!\left(x_0\right), t\right)\right]^{\!\T}\! + \epsilon^2\sigma\!\left(F_0^t\!\left(x_0\right), t\right)\sigma\!\left(F_0^t\!\left(x_0\right), t\right)^{\T}\!,
		\end{multlined}
		\label{eqn:pi_ode}
	\end{equation}
	subject to the initial condition \(\Pi(0) = \var{x}\).
	We show that the variance satisfies \cref{eqn:pi_ode} in \Cref{app:limit_moments_proof}.
\end{remark}

\subsection{Comparison to existing results}
\label{sec:comparison}
In this section we connect our work to the cognate bound derived by \citet{Sanz-AlonsoStuart_2017_GaussianApproximationsSmall}.
That paper considers the following SDE:
\begin{equation}\label{eqn:SAS_sde}
	\dif y_t^{(\epsilon)} = u\!\left(y_t^{(\epsilon)}\right)\dif t + \epsilon \, \tilde{\sigma}\dif W_t,
\end{equation}
where the diffusion coefficient \(\tilde{\sigma}\) is a constant matrix, which is a special case of \cref{eqn:sde_y}.
In this section, we apply our results to \cref{eqn:SAS_sde} to enable both bounds to be compared.
Note that \(\epsilon\) in our article is written as \(\sqrt{\epsilon}\) in \cite{Sanz-AlonsoStuart_2017_GaussianApproximationsSmall}; we will translate results from \cite{Sanz-AlonsoStuart_2017_GaussianApproximationsSmall} to use our notation, so that all results in this article are directly comparable.
In the following, \(c\) denotes an arbitrary finite and non-negative constant that can vary between inequalities.

Theorem 2.2 of \cite{Sanz-AlonsoStuart_2017_GaussianApproximationsSmall}, summarised, is as follows.
Let \(\xi_t^{(\epsilon)}\) be the probability measure associated with \(y_t^{(\epsilon)}\) (as defined in \cref{eqn:SAS_sde}), and let \(\nu_t^{(\epsilon)}\) be the probability measure associated with the corresponding linearisation \(l_t^{(\epsilon)}\) (as defined in \cref{sec:theory}).
Then there exists a constant \(c\) such that the Kullback--Leibler (KL) divergence \(D_{\mathrm{KL}}\) between \(\xi_t^{(\epsilon)}\) and \(\nu_t^{(\epsilon)}\) is bounded;
\begin{equation}
	\label{eqn:SAS}
	D_{\mathrm{KL}}\!\left(\xi_t^{(\epsilon)} \,\middle|\middle|\, \nu_t^{(\epsilon)}\right) \le D_{\mathrm{KL}}\!\left(\xi_0^{(\epsilon)} \,\middle|\middle|\, \nu_0^{(\epsilon)}\right) + c \, \epsilon^2\;.
\end{equation}
To focus on the scaling with \(\epsilon\), assume a fixed initial condition with \(D_{\mathrm{KL}}\!\left(\xi_0^{(\epsilon)} \,\middle|\middle|\, \nu_0^{(\epsilon)}\right) = 0\) (and \(\delta_r = 0\) in our bound \cref{eqn:main_ineq}). Then, employing the Hellinger distance \(D_{\mathrm{H}}\), \cref{eqn:SAS} implies
\begin{align}
	\label{eqn:convert}
	\norm{\avg{ y_t^{(\epsilon)} - l_t^{(\epsilon)}} } \le c D_{\mathrm{H}}\!\left(\xi_t^{(\epsilon)} ,\, \nu_t^{(\epsilon)}\right) \le c \sqrt{D_{\mathrm{KL}}\!\left(\xi_t^{(\epsilon)} \,\middle|\middle|\, \nu_t^{(\epsilon)}\right)} \le c \, \epsilon \;,
\end{align}
while our result \cref{eqn:main_ineq} and Jensen's inequality imply
\[
	\norm{\avg{ y_t^{(\epsilon)} - l_t^{(\epsilon)}} } \le \avg{\norm{y_t^{(\epsilon)} - l_t^{(\epsilon)}}} \leq c \, \epsilon^2\;.
\]
Thus, our bound on the moments is quadratic in \( \epsilon \) rather than linear.
If our conversion in \cref{eqn:convert} was optimal, then our approach in this article
provides a sharper bound on \(\norm{\avg{y_t^{(\epsilon)} - l_t^{(\epsilon)}}}\) that the
results of  \cite{Sanz-AlonsoStuart_2017_GaussianApproximationsSmall} imply, and do so for a
more general $ \sigma $.
The results in  \cite{Sanz-AlonsoStuart_2017_GaussianApproximationsSmall} on the KL divergence would be more natural in information-theoretic contexts, and our hope is that our explicit bound on the moments would be similarly preferred in other contexts.

\subsection{Gaussian initial condition}\label{sec:theory_gauss}
We now briefly consider the case when the initial condition follows a Gaussian distribution, i.e. \(x \isGauss{\mu_0, \Sigma_0}\), where \(\mu_0 \in \R^n\) and \(\Sigma_0 \in \R^{n \times n}\) are fixed and specified.
The linearisation then follows a Gaussian distribution itself, which is entirely characterised by the mean and covariance matrix described in \Cref{cor:limit_moments}.
Alternatively, these moments can be conveniently computed by simultaneously solving \cref{eqn:ode_det} for the state variable and \cref{eqn:pi_ode} for the linearised covariance.

A natural choice of reference point \(x_0\) is the mean of the initial Gaussian density, i.e. \(x_0 = \mu_0\).
The \(L_r\) distance between \(x\) and the mean \(\mu_0\) can be bounded by the trace of \(\Sigma_0\); for example, one such bound is
\begin{equation}\label{eqn:gauss_dist_bound}
	\delta_r^{r} \leq n^{3r/2 - 1} M_r \mathrm{tr}\left(\Sigma_0\right)^{r/2}, \quad M_r \coloneqq \frac{2^{r/2}\Gamma\!\left(\frac{r + 1}{2}\right)}{\sqrt{\pi}},
\end{equation}
where \(\Gamma\) denotes the Gamma function, with equality when \(n = 1\).
The initial covariance \(\Sigma_0\) directly measures the uncertainty in the initial condition, and we see through \cref{eqn:gauss_dist_bound} that as the components of \(\Sigma_0\) approach zero, the contribution of the initial uncertainty to the linearisation error in \cref{eqn:main_ineq} approaches zero also.
The linearised solution is then
\[
	l_t^{(\epsilon)} \isGauss{F_0^t\!\left(x_0\right), \, \nabla F_0^t\!\left(x_0\right) \Sigma_0\left[\nabla F_0^t\!\left(x_0\right)\right]^{\T} + \varepsilon^2 \Sigma_0^t\!\left(x_0\right)},
\]
where \(\Sigma_0^t\!\left(x_0\right)\) is given explicitly by
\begin{equation}\label{eqn:sigma_def}
	\Sigma_0^t\!\left(x_0\right) = \nabla F_0^t\!\left(x_0\right)\left(\int_0^t{L\!\left(x_0, \tau\right)L\!\left(x_0, \tau\right)^{\T} \dif\tau}\right)\left[\nabla F_0^t\!\left(x_0\right)\right]^{\T},
\end{equation}
and is the solution to the matrix differential equation \cref{eqn:pi_ode} in \Cref{rem:cov_ode}, subject to \(\Sigma_0^0\!\left(x_0\right) = O\), the \(n \times n\) zero matrix. 
The covariance matrix \(\Sigma_0^t\!\left(x_0\right)\) characterises the contribution of the ongoing uncertainty in the stochastic system.
The full covariance matrix \(\var{l_t^{(\epsilon)}}\) is also the solution to \cref{eqn:pi_ode} subject to the initial condition \(\Pi(0) = \var{x}\).
By jointly solving \cref{eqn:ode_det} for the deterministic trajectory (the mean of \(l_t^{(\epsilon)}\)) and \cref{eqn:pi_ode} for the covariance matrix, one can easily compute the linearised solution, describing exactly the assumed Gaussian approximation presented in \citet{SarkkaSolin_2019_AppliedStochasticDifferential}, and the dynamics linearisation used in the extended Kalman filter \cite{Jazwinski_2014_StochasticProcessesFiltering}.

\subsection{Fixed initial condition}\label{sec:theory_fixed}
Consider when the initial condition \(x\) is itself a fixed and known deterministic value, in which case we take \(x = x_0\) and \(\delta_r = 0\) for all \(r\).
In this situation, the bound \cref{eqn:main_ineq} on the linearisation error reduces to
\begin{equation}
	\avg{\norm{y_t^{(\epsilon)} - l_t^{(\epsilon)}}^r} \leq \left(K_{\nabla\nabla u}^r + K_{\nabla\sigma}^r\right)D_1\!\left(r,t, K_{\nabla u}, K_\sigma\right)\epsilon^{2r}.
	\label{eqn:main_ineq_fixed}
\end{equation}
We can consider the linearisation as equivalently arising from a first-order power series expansion of \(y_t^{(\epsilon)}\) in the noise-scale parameter \(\epsilon\), i.e.
\[
	y_t^{(\epsilon)} = F_0^t\!\left(x_0\right) + \epsilon z_t^{(\epsilon)} + R_2\left(x,t,\epsilon\right).
\]
where \(z_\epsilon \coloneqq \left(l_{t}^{(\epsilon)} - F_0^t\!\left(x_0\right)\right) / \epsilon\) is the first order term and \(R_2\) is a random quantity capturing the remaining deviation between \(y_t^{(\epsilon)}\) and the linearisation.
By rearranging and taking \(r = 1\) in \cref{eqn:main_ineq_fixed}, we therefore have the explicit Taylor-like bound
\[
	\frac{\avg{\norm{R_2\left(x,t,\epsilon\right)}}}{\epsilon^2} \leq \left(K_{\nabla \nabla u} + K_{\nabla\sigma}\right)D_1\!\left(1,t, K_{\nabla u}, K_\sigma\right),
\]
This result is consistent with the formulation of the linearisation error bounds by \citet{Blagoveshchenskii_1962_DiffusionProcessesDepending} and \citet{FreidlinWentzell_1998_RandomPerturbationsDynamical}, for instance.
Moreover, the distribution of the linearisation solution \cref{eqn:linear_sol} is Gaussian, which through \Cref{cor:limit_moments} we can again explicitly characterise in terms of the deterministic system, namely
\begin{equation}
	l_t^{(\epsilon)} \isGauss{F_0^t\!\left(x_0\right), \epsilon^2 \Sigma_0^t\!\left(x_0\right)},
	\label{eqn:linear_gauss_sol}
\end{equation}
where \(\Sigma_0^t\!\left(x_0\right)\) is defined in \cref{eqn:sigma_def}.
The distribution can be computed \emph{entirely} from the solution behaviour of the deterministic equation \cref{eqn:ode_det} and prior specification of \(\sigma\).
In \Cref{sec:theory_s2}, we demonstrate an application of these results to extend stochastic sensitivity \cite{Balasuriya_2020_StochasticSensitivityComputable} to arbitrary dimension.

\section{Extending stochastic sensitivity}\label{sec:theory_s2}
The results of \Cref{sec:theory_fixed} for a fixed initial condition provide a direct extension of the stochastic sensitivity tools first introduced by \citet{Balasuriya_2020_StochasticSensitivityComputable} for the fluid flow context.
Here, the deterministic model \cref{eqn:ode_det} is seen as a ``best-available'' model for the evolution of Lagrangian trajectories, and the driving vector field \(u\) is the Eulerian velocity of the fluid.
Stochastic sensitivity ascribes a scalar value to each deterministic trajectory by computing a maximum variance of projected deviation \cite{Balasuriya_2020_StochasticSensitivityComputable}.
The aim is to provide a \emph{single} computable number for each deterministic trajectory quantifying the impact of uncertainty in the velocity, independent of the scale (\(\epsilon\)) of the noise.
The natural restating of this original definition of stochastic sensitivity \cite{Balasuriya_2020_StochasticSensitivityComputable} in the $ n $-dimensional setting is as follows:

\begin{definition}[Stochastic sensitivity in \(\R^n\)]\label{def:ss_Rn}
	The \emph{stochastic sensitivity} is a scalar field \(S^2: \R^n \times [0,T] \to \left[0, \infty\right)\) given by
	\begin{equation*}
		S^2\!\left(x_0,t\right) \coloneqq \lim_{\epsilon\downarrow 0}\sup\set{\frac{1}{\epsilon}\var{p^{\T}\left(y_t^{(\epsilon)} - F_0^t\!\left(x_0\right)\right)} \,: \, p \in \R^n, \, \norm{p} = 1}.
	\end{equation*}
\end{definition}

\Cref{def:ss_Rn} is in the spirit of principal components analysis \cite{Jolliffe_2002_PrincipalComponentAnalysis}, performing a dimension reduction by projecting onto the direction in which the variance is maximised, thus capturing the most uncertainty in the data with a scalar value.
The anisotropic uncertainty in two-dimensions \cite{Balasuriya_2020_StochasticSensitivityComputable} is the direction-dependent projection (prior to optimising over all directions in \Cref{def:ss_Rn}).
Explicit theoretical expressions for both the stochastic sensitivity and the anisotropic sensitivity in two dimensions were obtained by \citet{Balasuriya_2020_StochasticSensitivityComputable}; these allowed for quantifying certainty in eventual trajectory locations without having to perform stochastic simulations.
We show here that our results in \(n\)-dimensions are a generalisation of the two-dimensional ones in \cite{Balasuriya_2020_StochasticSensitivityComputable}, which moreover establish Gaussianity as well as an explicit expression for the uncertainty measure.
A theoretically pleasing and computable expression for the stochastic sensitivity is obtainable;

\begin{theorem}[Computation of \(S^2\)]\label{thm:s2_calculation}
	For any \(x_0 \in \R^n\) and \(t \in [0,T]\),
	\begin{equation}
		S^2\!\left(x_0,t\right) = \norm{\Sigma_0^t\!\left(x_0\right)},
		\label{eqn:s2_calculation}
	\end{equation}
	where the covariance matrix \(\Sigma_0^t\) is defined in \cref{eqn:sigma_def}.
	Equivalently, \(S^2\!\left(x_0,t\right)\) is given by the maximum eigenvalue of \(\Sigma_0^t\!\left(x_0\right)\).
\end{theorem}
\begin{proof}
	See \Cref{app:s2_calculation_proof}.
	This result uses \Cref{thm:main} to establish the convergence of the covariance matrices, and then the properties of the spectral norm to establish \cref{eqn:s2_calculation}.
\end{proof}

Independent of the fluid mechanics context, \Cref{thm:s2_calculation} indicates that even for general systems, the matrix norm of \(\Sigma_0^t(x)\), i.e., the stochastic sensitivity \(S^2(x,t)\), can be used as {\em one} number which encapsulates the uncertainty of an initial state \(x\) after \(t\) time units.
The significance of this result is that the stochastic sensitivity has here been recovered as the maximal eigenvalue of a covariance matrix that is ubiquitous in the literature.
Stochastic sensitivity was formerly defined in \cite{Balasuriya_2020_StochasticSensitivityComputable} as the maximal value of the anisotropic uncertainty of a particular stochastic flow, and the connection to a linearisation was not apparent.

The stochastic sensitivity field can be calculated given any velocity data \(u\), and through the explicit expression \cref{eqn:sigma_def} for \(\Sigma_0^t\) can even be computed from only flow map data.
Computation does not require knowledge of the noise scale \(\epsilon\), so the \(S^2\) field is intrinsic in capturing the impact of the model dynamics on uncertainty, and any specified non-uniform diffusivity.
It has already been shown that, in the fluid flow context, stochastic sensitivity can identify coherent regions in two-dimensions \cite{BadzaEtAl_2023_HowSensitiveAre, Balasuriya_2020_StochasticSensitivityComputable}.
A simple approach is to define robust sets, which are those initial conditions for which the corresponding \(S^2\) value, i.e., the uncertainty in eventual location, are below some specified threshold.
This threshold can be defined precisely in terms of a spatial lengthscale of interest and the advective and diffusive characteristics of the flow, as Definition 2.9 of \cite{Balasuriya_2020_StochasticSensitivityComputable}.
Such a definition extends to the \(n\)-dimensional case as presented here, moreover establishing an easily computable method for determining coherent sets from the covariance matrix \(\Sigma_0^t\).

\section{Numerical validation \& examples}\label{sec:numerics}
This section will validate the theory presented in \Cref{sec:theory,sec:theory_s2}, by considering three example SDEs each leading to a different form of the strong error bound \cref{eqn:main_ineq}.
For each example, we first demonstrate heuristically that the solution converges to the limiting distribution described by \Cref{thm:limit_sol}, and then verify the error bound in \Cref{thm:main} directly by considering a range of values for the noise scale \(\epsilon\) and initial condition uncertainty \(\delta_r\).
We demonstrate numerically the form of the bound on the linearisation error predicted by \Cref{thm:main} is sharp, in the sense that estimates of the error scale exactly with the initial uncertainty \(\delta_r\) and ongoing uncertainty \(\epsilon\) as predicted.

All simulations in this section were generated using the Julia programming language \cite{BezansonEtAl_2017_JuliaFreshApproach}, with the implementations of numerical ODE and SDE solvers provided by the DifferentialEquations.jl package \cite{RackauckasNie_2017_DifferentialEquationsJlPerformant}.
The code is available at \href{https://github.com/liamblake/explicit-characterisation-sde-linearisation}{github.com/liamblake/explicit-characterisation-sde-linearisation}.

\subsection{Nonlinear dynamics, additive noise}\label{sec:numerics_nonlinear}
Consider the following SDE in 1D;
\begin{equation}\label{eqn:sine_sde}
	\dif y_t^{(\epsilon)} = \sin\!\left(y_t^{(\epsilon)}\right)\dif t + \epsilon \dif W_t.
\end{equation}
The deterministic system corresponding to \cref{eqn:sine_sde} has solution
\[
	F_0^t\!\left(x_0\right) = 2\arctan\left(e^{-t}\tan\left(\frac{x_0}{2}\right)\right).
\]
Further details of this example, including computation of the derivatives required in the linearisation, are provided in the supplementary material.

To explore the impact of initial condition uncertainty, we consider the univariate Gaussian initial condition \(y_0 = x \sim \Gauss{\mu, \rho^2}\), where the mean \(\mu\) is specified and the standard deviation \(\rho\) is a non-negative scaling parameter.
We linearise \cref{eqn:sine_sde} about the deterministic trajectory \(F_0^t\!\left(\mu\right)\) originating from the mean, that is, \(\mu\) is the chosen reference point.
This ensures that for any \(r \geq 0\)
\begin{equation}\label{eqn:num_gauss_init}
	\delta_r^r = \avg{\abs{x - \mu}^r} = M_r \rho^r.
\end{equation}
where \(M_r\) is as defined in \cref{eqn:gauss_dist_bound}.
This property of the univariate Gaussian distribution allows us to easily control the uncertainty in the initial condition and verify the bounds; by sending the parameter \(\rho\) to zero we ensure that \(\delta_r\) approaches zero also.

The linearised equation is then
\begin{equation}
	\dif l_t^{(\epsilon)} = \left[F_0^t\!\left(\mu\right) + \cos\!\left(F_0^t\!\left(\mu\right)\right)\left(l_t^{(\epsilon)} - F_0^t\!\left(\mu\right)\right)\right]\dif t + \epsilon \dif W_t, \quad l_0^{(\epsilon)} \isGauss{\mu, \rho^2}.
	\label{eqn:sine_linear}
\end{equation}
and the solution follows a Gaussian distribution, specifically
\begin{equation}\label{eqn:num_linear_sol}
	l_t^{(\epsilon)} \isGauss{F_0^t\!\left(\mu\right), \, \rho^2\nabla F_0^t\!\left(\mu\right)^2 + \epsilon^2\Sigma_0^t\!\left(\mu\right)}.
\end{equation}
where \(\Sigma_0^t\!\left(\mu\right)\) is computed by solving \cref{eqn:pi_ode} numerically subject to a zero initial condition.

In this example, we take \(\mu_0 = 0.5\) and consider the solutions of \cref{eqn:sine_sde} and \cref{eqn:sine_linear} at time \(t = 1.5\).
We generate accurate samples of \cref{eqn:sine_sde} and \cref{eqn:sine_linear} jointly (i.e. using the same numerical realisations of the Wiener process \(W_t\)) using the stochastic Runge-Kutta scheme SRI \cite{Rossler_2010_RungeKuttaMethodsStrong} with an adaptive step size \cite{RackauckasNie_2017_AdaptiveMethodsStochastic}.

\begin{figure}
	\begin{center}
		\includegraphics[width=\textwidth]{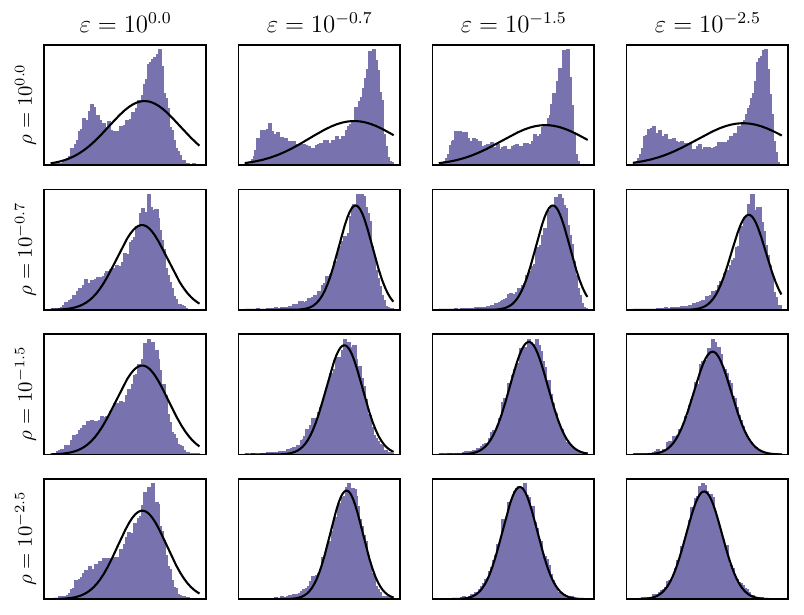}
		\caption{Histograms of stochastic samples of \cref{eqn:sine_sde}, subject to the Gaussian initial condition \cref{eqn:num_gauss_init}, for varying initial uncertainty scale \(\rho\) and ongoing uncertainty scale \(\epsilon\).
			The distribution of the corresponding solution \cref{eqn:num_linear_sol} to the linearised equation is overlaid in black.}
		\label{fig:sine_hists}
	\end{center}
\end{figure}

In \Cref{fig:sine_hists}, we show histograms of \(N = 10000\) samples of the solution to nonlinear SDE \cref{eqn:sine_sde} and the corresponding probability density function of the linearised solution \cref{eqn:num_linear_sol}, for different combinations of \(\epsilon\) and \(\rho\).
Even when the ongoing noise is small, the nonlinearity of the drift term means that a large initial uncertainty results in a non-Gaussian distribution.
However, in situations where both the initial and ongoing uncertainties are small, the Gaussian solution to the linearised equation provides a reasonable approximation. 
In the limit of both small initial (\(\rho \to 0\)) and small ongoing (\(\epsilon \to 0\)) uncertainty (towards the bottom right), we see that the distribution of the samples approach the Gaussian density of the linearisation solution, matching the understanding that the linearisation approximation is ``reasonable'' for small noise regimes.

Since the drift term is nonlinear and the noise is additive in \cref{eqn:sine_sde}, the bound predicted by \Cref{thm:main} has the form
\[
	\avg{\norm{y_t^{(\epsilon)} - l_t^{(\epsilon)}}^r} \leq D_1\!\left(r,t, K_{\nabla u}, K_\sigma\right)\epsilon^{2r} + M_{2r}D_2\!\left(r,t, K_{\nabla u}\right)\rho^{2r}.
\]
where we have taken \(K_{\nabla\nabla u} = 1\) and \(K_{\nabla\sigma} = 0\).
To numerically validate this bound under the Gaussian initial condition \cref{eqn:num_gauss_init}, define for \(r \geq 1\) the error measure
\begin{equation}
	E_r\!\left(\epsilon, \rho\right) \coloneqq \frac{1}{N}\sum_{i=1}^N{\norm{\hat{y}_{i}^{(\epsilon)} - \hat{l}_i^{(\epsilon)}}^r},
	\label{eqn:strong_err_mc_estimate}
\end{equation}
which is a Monte-Carlo estimator of the right-hand side of \cref{eqn:main_ineq}, where \(\hat{y}_1^{(\epsilon)},\dotsc, \hat{y}_N^{(\epsilon)}\) and \(\hat{l}_1^{(\epsilon)},\dotsc, \hat{l}_N^{(\epsilon)}\) are \(N\) numerical samples of the solutions to SDE \cref{eqn:sde_y} and the linearisation \cref{eqn:linear_sde_inform} respectively.

We directly validate the \emph{form} of the error bound (as a function of \(\epsilon\) and \(\rho\)) in \Cref{fig:sine_delta_eps_lines}, by computing \(E_1\) using samples for each pair of \(\epsilon\) and \(\rho\) values.
In \Cref{fig:sine_eps_lines}, we demonstrate the relationship between \(E_1\) and the ongoing uncertainty \(\epsilon\) for several different fixed values of \(\rho\), each corresponding to a different colour.
A least squares estimate of a line of best fit of the form \(E_1 = \beta_0 + \beta_1 \epsilon^2 \), for fixed coefficients \(\beta_0\) and \(\beta_1\), is fitted to the observed errors (in untransformed space) to verify the scaling of our bound in \Cref{thm:main}.
We see that the line of best fit accurately matches the observed values of \(E_1\), verifying that \(E_1\) is in fact scaling with \(\epsilon^2\) as predicted.
\Cref{fig:sine_delta_lines} provides a similar demonstration between \(E_1\) and the initial uncertainty \(\rho\), where now each colour corresponds to a different fixed value of \(\epsilon\).
We again fit lines of the form \(E_1 = \beta_0 + \beta_1 \rho^2\) to verify the scaling of the bound, and see that the lines match the observed values of \(E_2\).
Thus, we have also validated that \(E_1\) scales with \(\rho^2\), as expected from \Cref{thm:main}.

\begin{figure}
	\begin{center}
		\begin{subfigure}{\textwidth}
			\includegraphics[width=\textwidth]{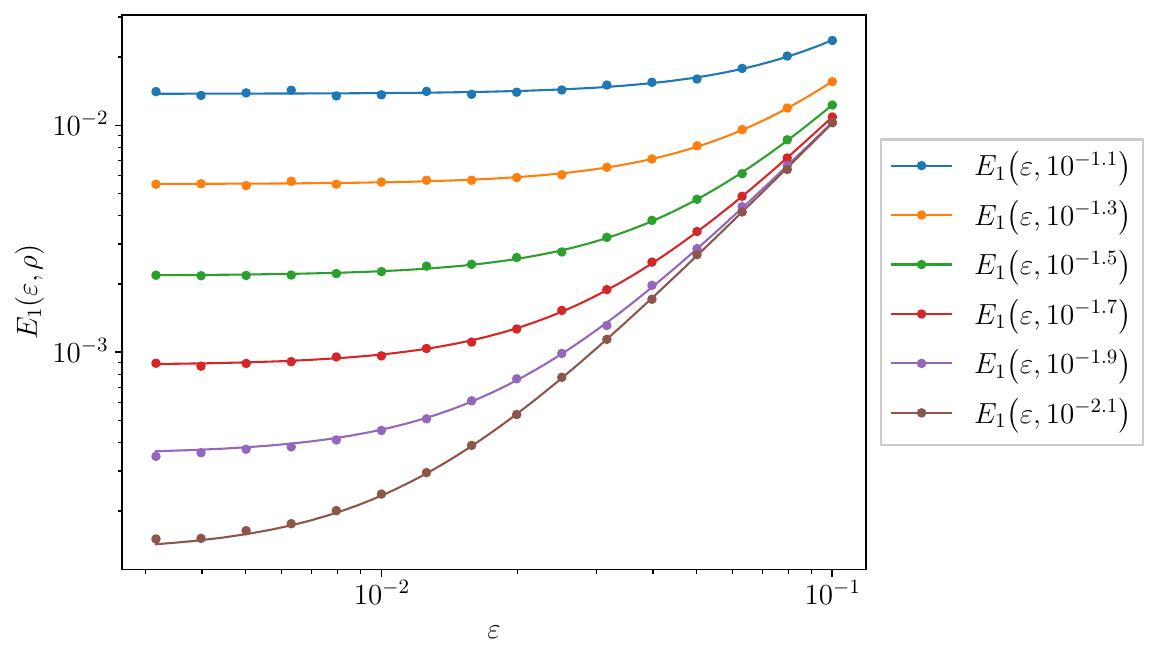}
			\caption{Estimates of the strong error (with \(r = 1\)) in linearising \cref{eqn:sine_sde} with \cref{eqn:sine_linear}, for varying ongoing uncertainty parameter \(\epsilon\).
				Each colour corresponds to a different value of the initial uncertainty parameter \(\rho\).
				A (least squares) line of best fit of the form \(\beta_0 + \beta_1 \epsilon^2\) is included in the corresponding colour.}
			\label{fig:sine_eps_lines}
		\end{subfigure}
		\begin{subfigure}{\textwidth}
			\includegraphics[width=\textwidth]{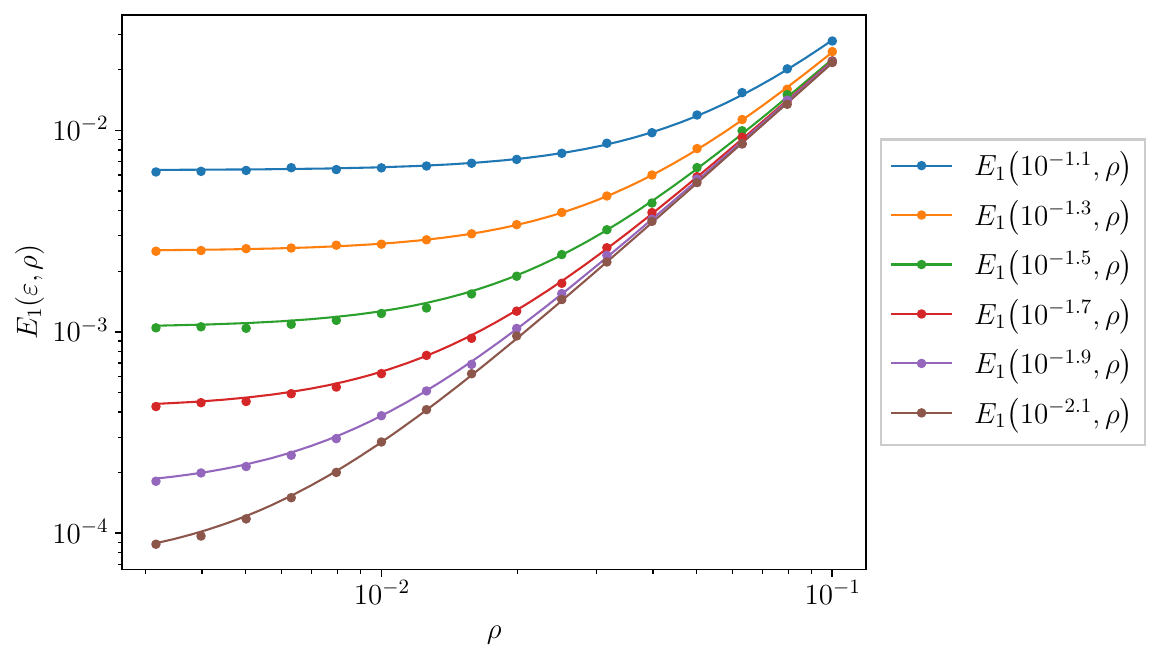}
			\caption{Estimates of the strong error (with \(r = 1\)) for varying initial uncertainty parameter \(\rho\).
				Each colour corresponds to a different value of the ongoing uncertainty parameter \(\epsilon\).
				A (least squares) line of best fit of the form \(\beta_0 + \beta_1 \rho^2\) is included in the corresponding colour.}
			\label{fig:sine_delta_lines}
		\end{subfigure}
		\caption{Validation of the theoretical bound predicted by \Cref{thm:main}, when \(r = 1\), on numerical realisations of the solution to the 1D example \cref{eqn:sine_sde}.}
		\label{fig:sine_delta_eps_lines}
	\end{center}
\end{figure}

\subsection{Linear dynamics, multiplicative noise}\label{sec:numerics_multiplicative}
Now consider the following SDE with multiplicative noise in 1D;
\begin{equation}
	\dif y_t^{(\epsilon)} = \frac12 y_t^{(\epsilon)}\dif t + \varepsilon \cos\!\left(y_t^{(\epsilon)}\right) \dif W_t.
	\label{eqn:1d_mult}
\end{equation}
The corresponding deterministic system is linear and has solution
\begin{equation}
	F_0^t\!\left(x_0\right) = \exp\!\left(\frac{t}{2}\right) x_0,
	\label{eqn:1d_mult_det_sol}
\end{equation}
with additional details provided in the supplementary materials.
As with the previous example in \Cref{sec:numerics_nonlinear}, we take the Gaussian initial condition \cref{eqn:num_gauss_init} with variance \(\rho^2\) and linearised \cref{eqn:1d_mult} about the initial mean \(\mu\).
The linearised equation is then
\begin{equation}
	\dif l_t^{(\epsilon)} = \frac12 l_t^{(\epsilon)}\dif t + \epsilon \cos\!\left(\exp\left(\frac{t}{2}\right)\mu\right) \dif W_t, \quad l_0^{(\epsilon)} \sim \Gauss{\mu, \rho^2},
	\label{eqn:1d_mult_linear}
\end{equation}
with Gaussian solution \cref{eqn:num_linear_sol}.
We take the initial point \(\mu = 2\) and consider the solutions at time \(t = 1\).
To generate numerical realisations of the solutions to \cref{eqn:1d_mult} and \cref{eqn:1d_mult_linear} with the same realisations of \(W_t\), we use the same set-up as in the previous example.

In \Cref{fig:sine_hists}, we show histograms of \(N = 10000\) samples of the multiplicative noise SDE \cref{eqn:1d_mult} and the corresponding probability density function of the linearised solution, for different combinations of \(\epsilon\) and \(\rho\).
We again see that in the limit of both small initial and small ongoing uncertainty (towards the bottom right), we see that the distribution of the samples approach the Gaussian density of the linearisation solution.

\begin{figure}
	\begin{center}
		\includegraphics[width=\textwidth]{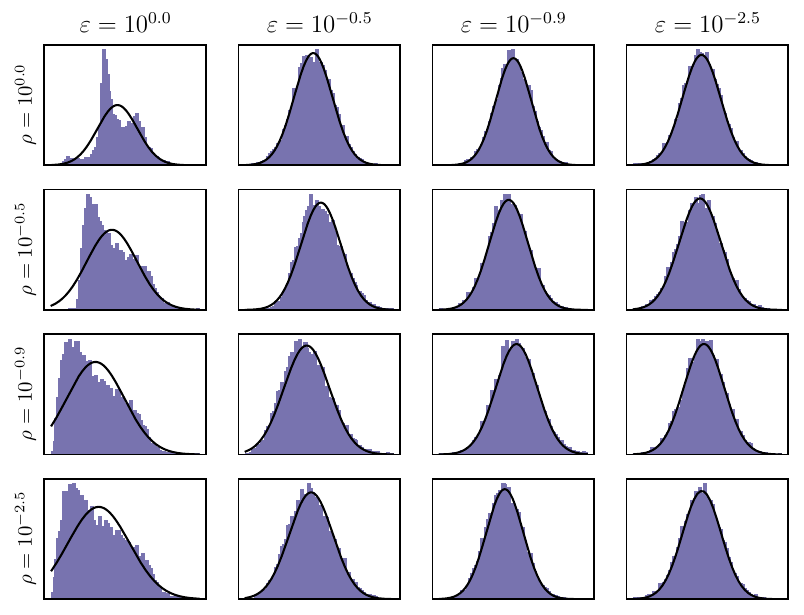}
		\caption{The same arrangement as \Cref{fig:sine_hists}, but for the 1D multiplicative noise SDE \cref{eqn:1d_mult}.}
		\label{fig:1d_mult_hists}
	\end{center}
\end{figure}

Since the drift term is linear and the noise multiplicative in \cref{eqn:1d_mult}, the bound predicted by \Cref{thm:main} has the form
\[
	\avg{\norm{y_t^{(\epsilon)} - l_t^{(\epsilon)}}^r} \leq D_1\!\left(r,t, K_{\nabla u}, K_\sigma\right)\epsilon^{2r} + M_{r}D_3\!\left(r,t, K_{\nabla u}\right)\epsilon^r\rho^{r},
\]
where we have \(K_{\nabla\nabla u} = 0\) and \(K_{\nabla\sigma} = 1\).
In \Cref{fig:multiplicative_delta_eps_lines}, we again validate the form of this bound (for \(r = 1\); results for additional values of \(r\) are provided in the supplementary material) by approximating the left-hand side with \(E_1\) computed from realisations of the solution to \cref{eqn:1d_mult} and the linearisation \cref{eqn:1d_mult_linear}.
For each fixed value of the initial uncertainty \(\rho\), in \Cref{fig:multiplicative_eps_lines}, we fit a line of best fit of the form \(\beta_1 \epsilon + \beta_2 \epsilon^2\) to validate that the strong error scales as predicted.
Similarly, in \Cref{fig:multiplicative_eps_lines} we fit a line of best fit of the form \(\beta_0 + \beta_1 \rho\) and confirm that the linearisation error follows this scaling.

\begin{figure}
	\begin{center}
		\begin{subfigure}{\textwidth}
			\includegraphics[width=\textwidth]{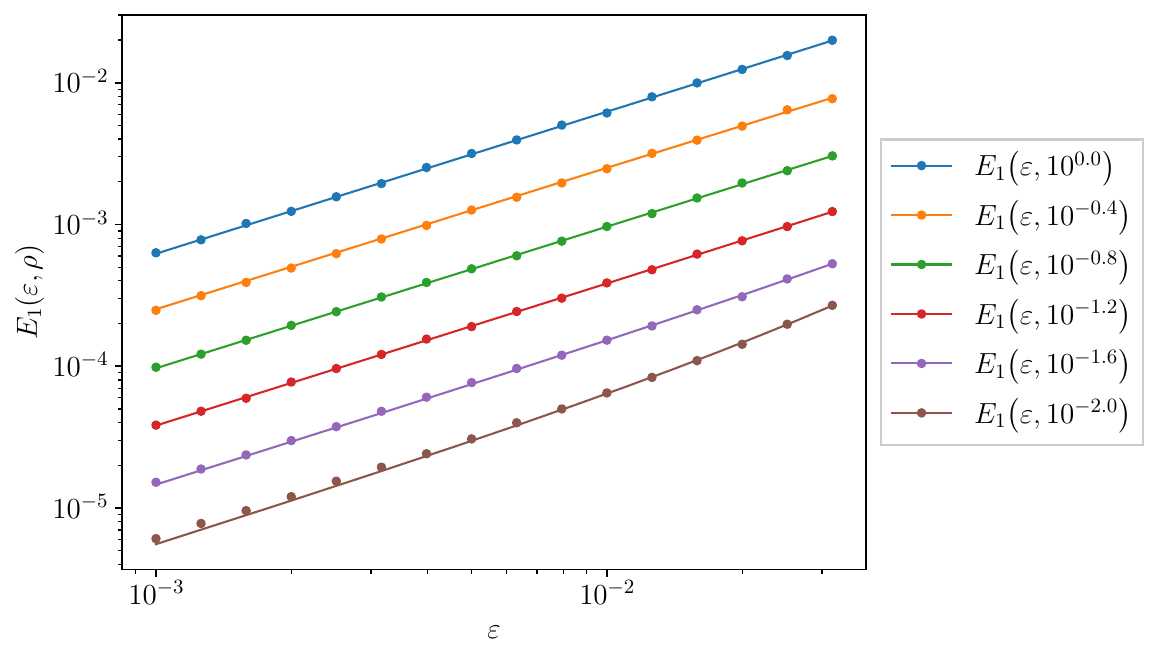}
			\caption{Estimates of the strong order (with \(r = 1\)) for varying ongoing uncertainty parameter \(\epsilon\).
				Each colour corresponds to a different value of the initial uncertainty parameter \(\rho\).
				A (least squares) line of best fit of the form \(\beta_1 \epsilon + \beta_2 \epsilon^2\) is included in the corresponding colour.}
			\label{fig:multiplicative_eps_lines}
		\end{subfigure}
		\begin{subfigure}{\textwidth}
			\includegraphics[width=\textwidth]{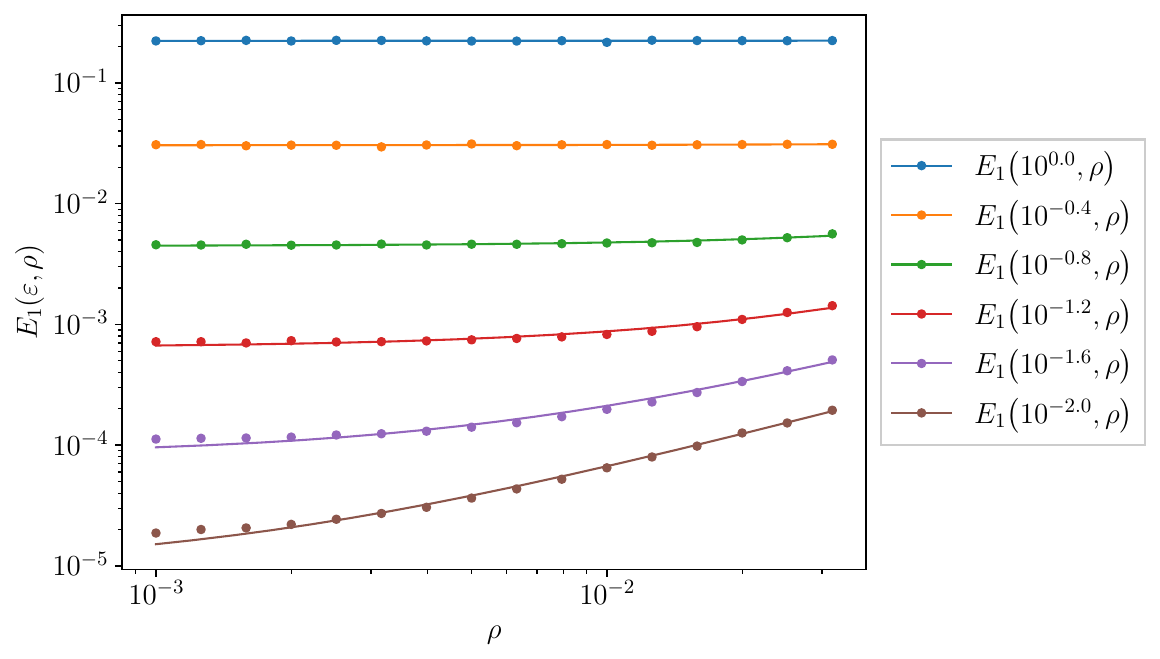}
			\caption{Estimates of the strong order (with \(r = 1\)) for varying initial uncertainty parameter \(\rho\).
				Each colour corresponds to a different value of the ongoing uncertainty parameter \(\epsilon\).
				A (least squares) line of best fit of the form \(\beta_0 + \beta_1 \rho\) is included in the corresponding colour.}
			\label{fig:multiplicative_delta_lines}
		\end{subfigure}
		\caption{Validation of the theoretical bound predicted by \Cref{thm:main}, when \(r = 1\), on numerical realisations of the solution to the 1D example \cref{eqn:1d_mult}.}
		\label{fig:multiplicative_delta_eps_lines}
	\end{center}
\end{figure}

\subsection{Fixed initial condition}\label{sec:numerics_2d}
In this example, we consider a two-dimensional model and a fixed initial condition, to validate the results presented in \Cref{sec:theory_fixed}.
Following the example in Chapter 5 of \cite{SamelsonWiggins_2006_LagrangianTransportGeophysical}, we consider an unsteady meandering jet in two dimensions, which may serve as an idealised model of geophysical Rossby waves \cite{Pierrehumbert_1991_ChaoticMixingTracer}.
The velocity field for \(y \equiv \left(y_1, y_2\right)^{\T}\) is given by
\begin{equation}
	u\!\left(y, t\right) = \begin{bmatrix}
		c - A\sin\!\left(Ky_1\right)\cos\!\left(y_2\right) + \oldepsilon_{\mathrm{mj}} l_1\sin\!\left(k_1\left(y_1 - c_1 t\right)\right)\cos\!\left(l_1 y_2\right) \\
		AK\cos\!\left(Ky_1\right)\sin\!\left(y_2\right) + \oldepsilon_{\mathrm{mj}} k_1\cos\!\left(k_1\left(y_1 - c_1t\right)\right)\sin\!\left(l_1 y_2\right)
	\end{bmatrix}.
	\label{eqn:jet_ex}
\end{equation}
The velocity field describes a kinematic travelling wave with deterministic oscillatory perturbations in a co-moving frame.
Here, \(A\) is the amplitude and \(c\) is the phase speed of the primary wave, and \(K\) is the wavenumber in the \(y_1\)-direction.
The oscillatory perturbation has amplitude \(\oldepsilon_{\mathrm{mj}}\), phase speed \(c_1\) (in the co-moving frame), and wavenumbers \(k_1\) and \(l_1\) in the \(y_1\)- and \(y_2\)-directions respectively.
Throughout, we take the parameter values \(c = 0.5\), \(A = 1\), \(K = 4\), \(l_1 = 2\), \(k_1 = 1\), \(c_1 = \pi\), and \(\oldepsilon_{\mathrm{mj}} = 0.3\).
For these values, the flow consists of a meandering jet with vortex structures within the meanders, and a chaotic zone which influences the fluid transfer between the jet and the vortices.

We introduce multiplicative noise by considering stochastic perturbations to the phase speed \(c\) and the primary amplitude \(A\), which we model with the respective components of a \(2\)-dimensional Wiener process \(W_t = \left(W_t^{(1)}, W_t^{(2)}\right)^{\T}\).
Then, we specify the diffusion term as
\begin{equation}
	\sigma\!\left(y,t\right) = \begin{bmatrix}
		1 & \sin\!\left(Ky_1\right)\cos\!\left(y_2\right)  \\
		0 & K\cos\!\left(Ky_1\right)\sin\!\left(y_2\right)
	\end{bmatrix}.
	\label{eqn:jet_ex_sigma}
\end{equation}

We consider the fixed initial condition \(x_0 = \left(0, 1\right)\) and the prediction of the model at time \(t = 1\).
We then consider a linearisation of the SDE about the deterministic trajectory \(F_0^t\!\left(x_0\right)\), where \(F_0^t\) is the deterministic flow map corresponding to the vector field \cref{eqn:jet_ex}.
To compute the Gaussian distribution \cref{eqn:linear_gauss_sol} of the linearised solution, we again solve \cref{eqn:pi_ode} numerically with initial condition \(\Sigma_0^t\!\left(x_0\right) = O\).
Specifically, \cref{eqn:pi_ode} is solved jointly with the deterministic state equation \cref{eqn:ode_det} using the hybrid method proposed by \citet{Mazzoni_2008_ComputationalAspectsContinuous}.
This hybrid method combines a Taylor-Heun approximation with a Gauss-Legendre one and ensures that the numerical solution of the covariance equation is symmetric and positive semi-definite while maintaining both accuracy and computational efficiency.

\Cref{fig:y_hists} shows the resulting simulations of \(y_t^{(\epsilon)}\) for four different values of \(\epsilon\).
The realisations are binned as a histogram and bin counts are normalised, to provide an empirical estimate of the probability density function of \(y_t^{(\epsilon)}\).
Superimposed (in solid black) are the first, second and third standard-deviation contours of the probability density function of the Gaussian distribution that solves the linearised equation.
The first three standard-deviation levels of the \(2\times 2\) sample covariance matrix of the realisations of \(y_t^{(\epsilon)}\), are also overlaid (in dashed blue).
As \(\epsilon\) decreases towards \(0\), the samples increasingly resemble a Gaussian distribution, and both the mean and covariance coincide with the corresponding limits.

\begin{figure}
	\begin{center}
		\begin{subfigure}{0.49\textwidth}
			\includegraphics[width=\textwidth]{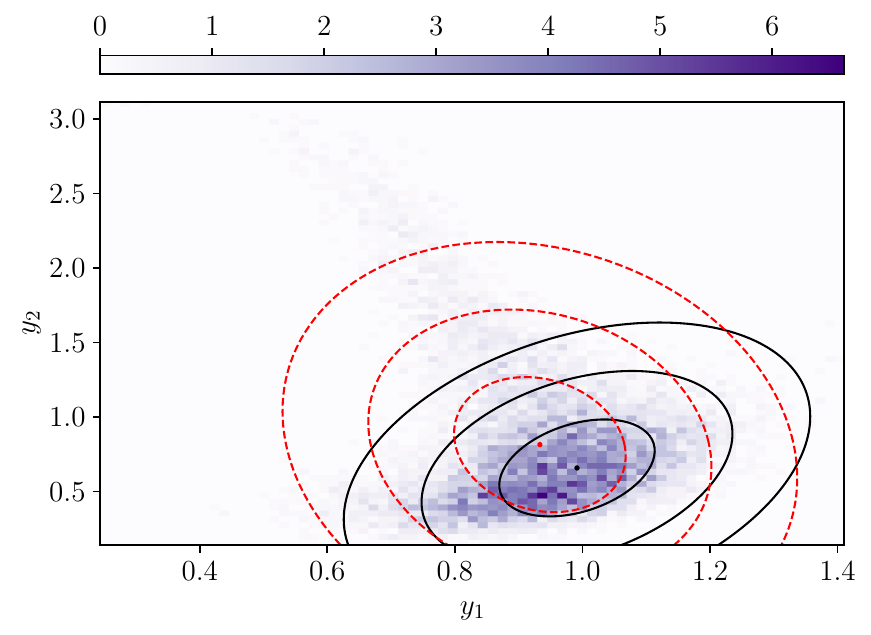}
			\caption{\(\epsilon = 10^{-1}\)}
			\label{fig:y_hists_a}
		\end{subfigure}
		\begin{subfigure}{0.49\textwidth}
			\includegraphics[width=\textwidth]{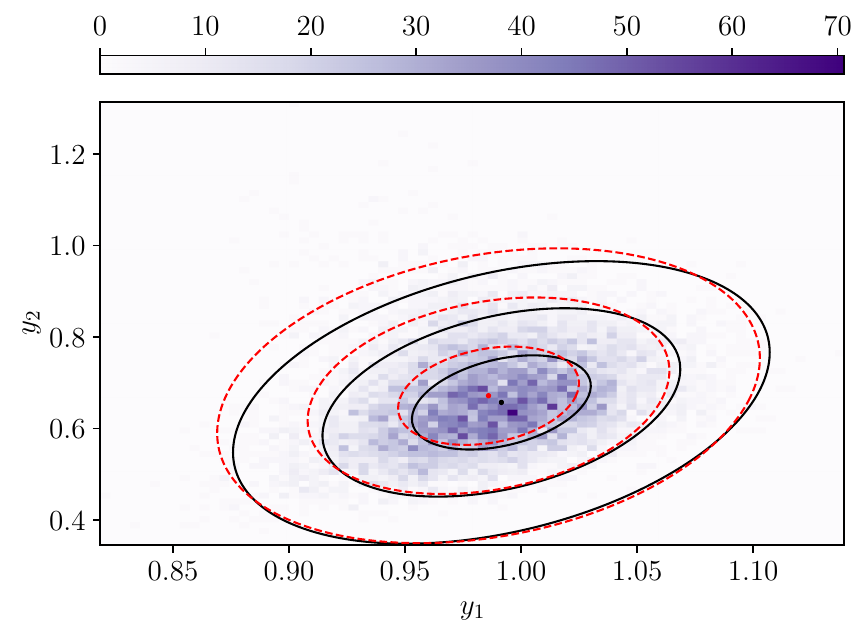}
			\caption{\(\epsilon = 10^{-1.5}\)}
			\label{fig:y_hists_b}
		\end{subfigure}
		\begin{subfigure}{0.49\textwidth}
			\includegraphics[width=\textwidth]{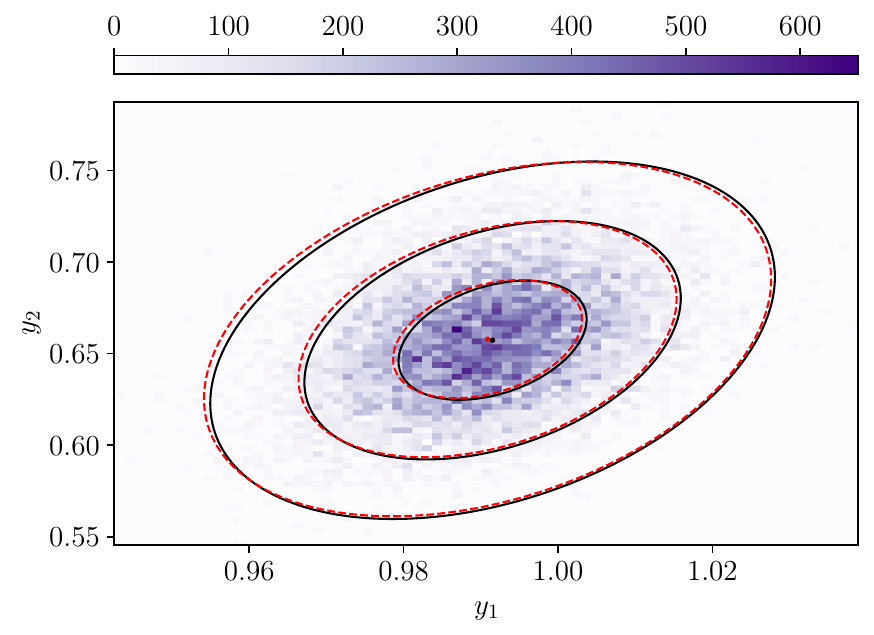}
			\caption{\(\epsilon = 10^{-2}\)}
			\label{fig:y_hists_c}
		\end{subfigure}
		\begin{subfigure}{0.49\textwidth}
			\includegraphics[width=\textwidth]{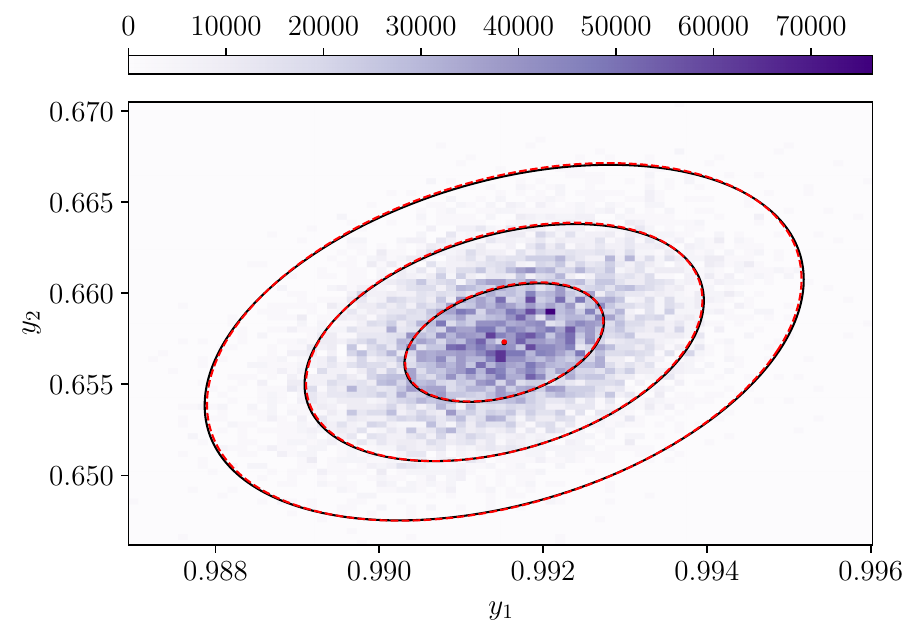}
			\caption{\(\epsilon = 10^{-3}\)}
			\label{fig:y_hists_d}
		\end{subfigure}
		\caption{Histograms of \(y_t^{(\epsilon)}\) from direct simulation of the SDE with drift \cref{eqn:jet_ex} and diffusivity \cref{eqn:jet_ex_sigma} subject to the fixed initial condition, for four different \(\epsilon\) values.
		Overlaid in black are contours of the Gaussian solution \cref{eqn:linear_gauss_sol} of the linearised SDE \cref{eqn:linear_sde_inform}, which correspond to the first three standard deviation levels centred at the mean \(F_0^t(x)\).
		In dashed blue are corresponding contours computed from the sample covariance matrix of the realisations.
		}
		\label{fig:y_hists}
	\end{center}
\end{figure}

For a fixed initial condition, \cref{eqn:main_ineq} predicts that the expected distance between the original SDE solution and that of a linearisation satisfies
\[
	\avg{\norm{y_t^{(\epsilon)} - l_t^{(\epsilon)}}^r} \leq \left(K_{\nabla\nabla u} + K_{\nabla\sigma}\right)D_1\!\left(r,t, K_{\nabla u}, K_\sigma\right)\epsilon^{2r}.
\]
To numerically estimate the left-hand side of \cref{eqn:main_ineq}, we again use a Monte-Carlo estimator;
\[
	E_r\!\left(\epsilon\right) \coloneqq \frac{1}{N}\sum_{i=1}^N{\norm{\hat{y}_i^{(\epsilon)} - \hat{l}_i^{(\epsilon)}}^r}.
\]
For \(r = 1,2,3,4\), \(E_r\!\left(\epsilon\right)\) is shown (in a logarithmic scale) for decreasing values of \(\epsilon\) in \Cref{fig:gamma_z_valid}.
\Cref{thm:main} predicts that \(\log_{10}\left(E_r\!\left(\epsilon\right)\right)\) should decay linearly with a slope greater than \(2r\) as \(\epsilon\) decreases to zero.
The least squares lines of best fit for each value of \(r\) in \Cref{fig:gamma_z_valid} show this behaviour, and are therefore consistent with \Cref{thm:main}.

\begin{figure}
	\begin{center}
		\begin{subfigure}{0.49\textwidth}
			\includegraphics[width=\textwidth]{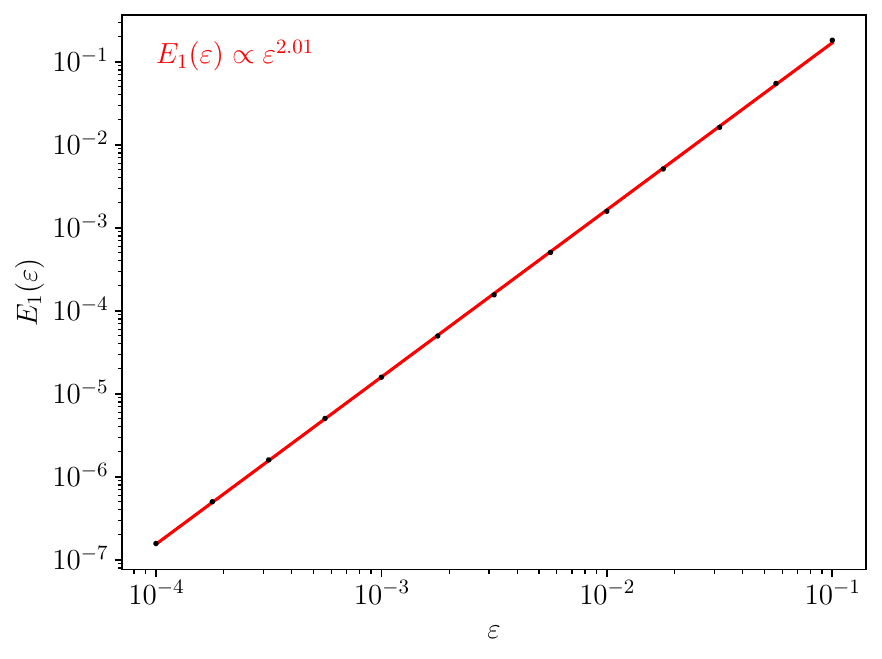}
			\caption{\(r = 1\) (mean)}
			\label{fig:gamma_z_valid_1}
		\end{subfigure}
		\begin{subfigure}{0.49\textwidth}
			\includegraphics[width=\textwidth]{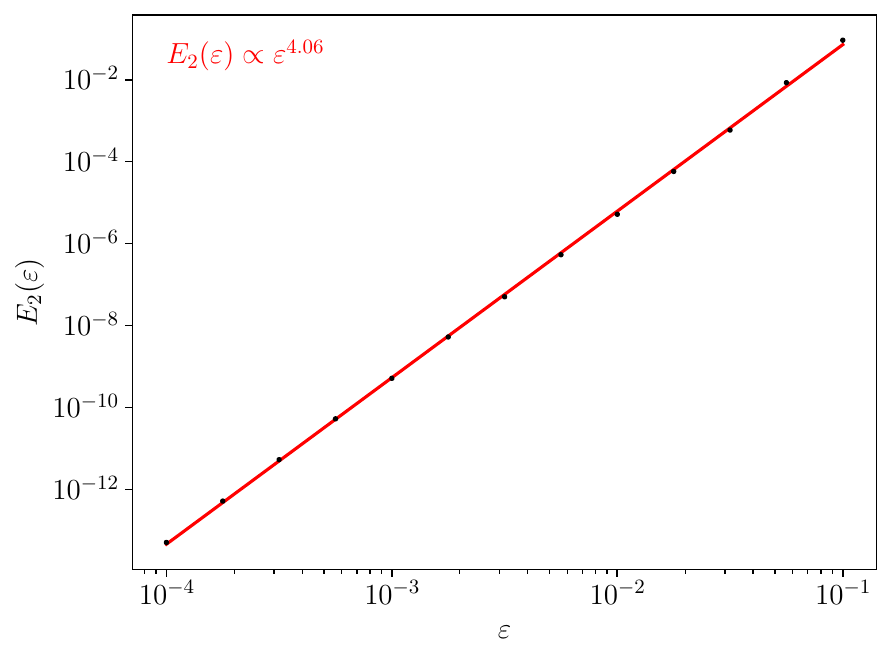}
			\caption{\(r = 2\) (variance)}
			\label{fig:gamma_z_valid_2}
		\end{subfigure}
		\begin{subfigure}{0.49\textwidth}
			\includegraphics[width=\textwidth]{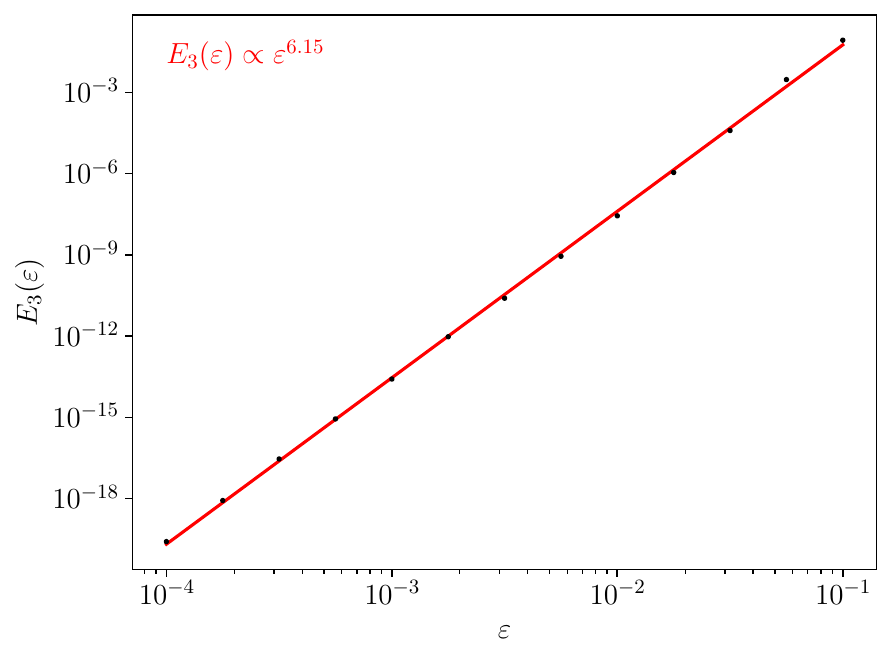}
			\caption{\(r = 3\) (skewness)}
			\label{fig:gamma_z_valid_3}
		\end{subfigure}
		\begin{subfigure}{0.49\textwidth}
			\includegraphics[width=\textwidth]{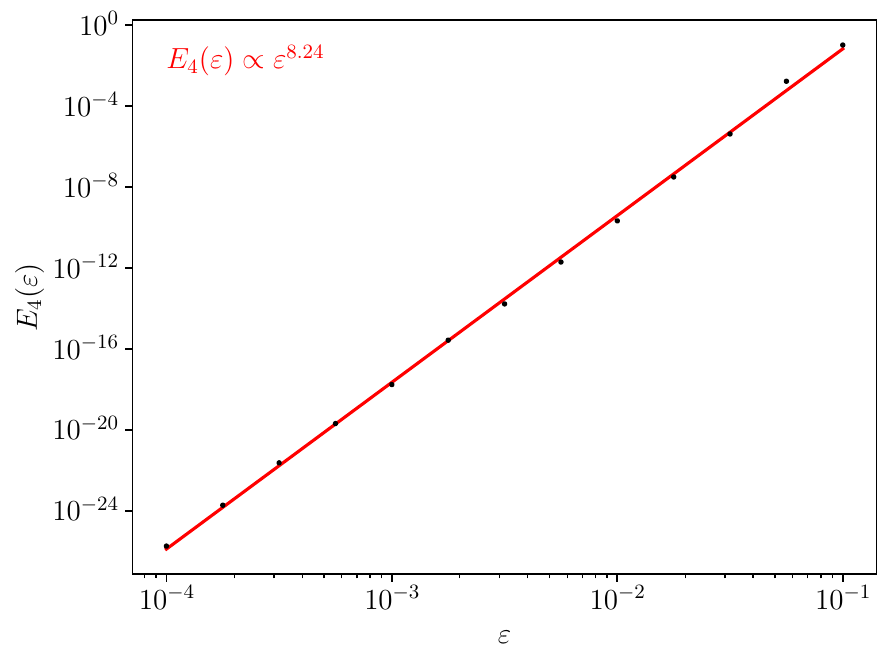}
			\caption{\(r = 4\) (kurtosis)}
			\label{fig:gamma_z_valid_4}
		\end{subfigure}

		\caption{Validation of \Cref{thm:main}, by plotting the sample \(r\)th raw moment distance (the error metric \(E_r(\epsilon)\)) between \(10000\) realisations of the meandering jet SDE and a corresponding linearisation, for decreasing values of \(\epsilon\).
			A line of best fit (in red) is placed on each, and the resulting slope indicated.}
		\label{fig:gamma_z_valid}
	\end{center}
\end{figure}

\subsection{Computing stochastic sensitivity} \label{sec:comput_s2}
In this section, we illustrate the computability of stochastic sensitivity as 
described in \Cref{thm:s2_calculation}.
We again consider the meandering jet \cref{eqn:jet_ex} with multiplicative noise specified by \cref{eqn:jet_ex_sigma}.
We take the same choice of parameters as in \cref{sec:numerics_2d}, except for the perturbation amplitude \(\oldepsilon_{\mathrm{mj}}\) which is varied to obtain qualitatively different behaviour in the system.
For each initial condition in a \(400 \times 400\) uniform grid on \(\left[0, \pi\right] \times \left[0, \pi\right]\), the \(S^2\) value is calculated using \cref{eqn:s2_calculation}.
\Cref{fig:s2_field} shows the resulting \(S^2\) field from time \(0\) to \(t = 1\), for two different values of \(\oldepsilon_{\mathrm{mj}}\).
The novel aspect of this section is that stochastic sensitivity is demonstrated to be computable with a straightforward spectral norm (or equivalently eigenvalue) calculation, rather than the computation involving three integrals originally presented in \cite{Balasuriya_2020_StochasticSensitivityComputable}.
With the additional extension of stochastic sensitivity to arbitrary dimensions, this computation enables a fast quantification of uncertainty in a dynamical system.

\begin{figure}
	\begin{center}
		\begin{subfigure}{0.49\textwidth}
			\includegraphics[width=\textwidth]{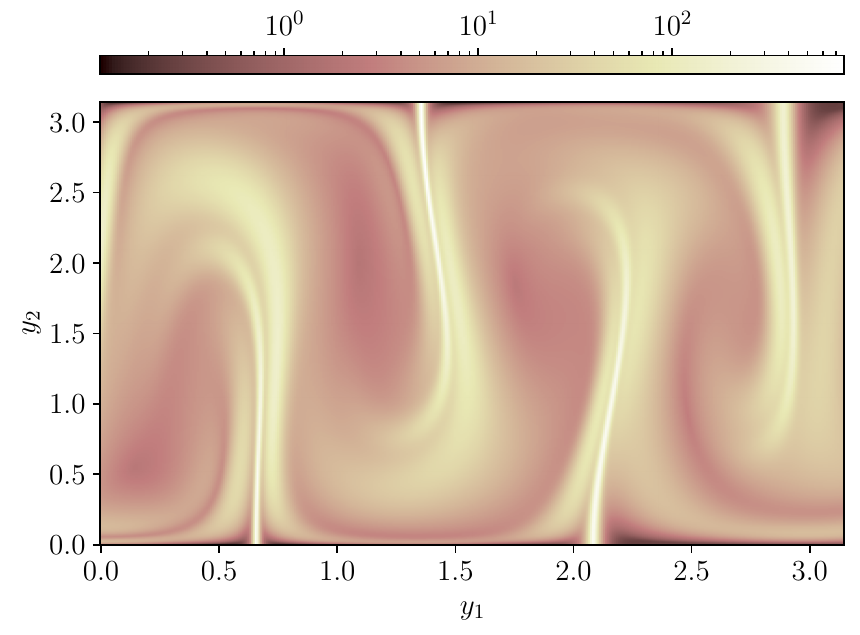}
			\caption{\(\oldepsilon_{\mathrm{mj}} = 0.3\)}
			\label{fig:s2_field_0.3}
		\end{subfigure}
		\begin{subfigure}{0.49\textwidth}
			\includegraphics[width=\textwidth]{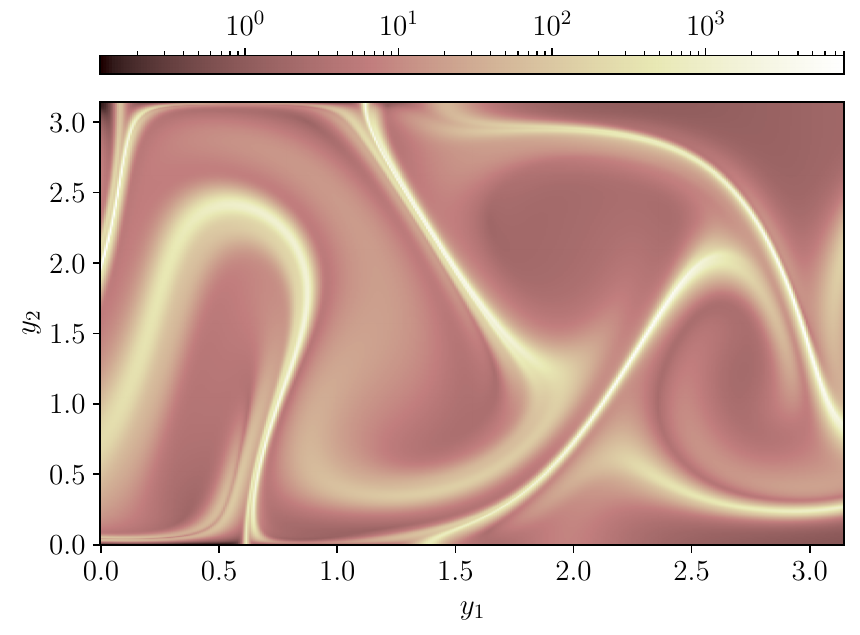}
			\caption{\(\oldepsilon_{\mathrm{mj}} = 1.0\)}
			\label{fig:s2_field_1.0}
		\end{subfigure}
		\caption{The \(S^2\) field of the meandering jet flow \cref{eqn:jet_ex} over the time interval \([0,1]\), for two different sets of parameters with qualitatively different behaviour.
			The \(S^2\) value for each initial condition is computed directly as the operator norm of the covariance matrix \(\Sigma_0^1\!\left(x_0\right)\), as per \cref{eqn:s2_calculation}.}
		\label{fig:s2_field}
	\end{center}
\end{figure}

The stochastic sensitivity field can highlight Lagrangian coherent structures within the flow, by identifying regions of the flow with a relatively small uncertainty, as measured by a \emph{single} number for each initial condition.
Subsets of the spatial domain corresponding to coherent regions can be extracted by taking a threshold on \(S^2\) as described in the original work \cite{Balasuriya_2020_StochasticSensitivityComputable}; examples of coherent structure extraction with stochastic sensitivity on both toy models and real data can be found in \cite{Balasuriya_2020_StochasticSensitivityComputable} and \cite{BadzaEtAl_2023_HowSensitiveAre}.

\section{Conclusions}\label{sec:discussion}
This paper has provided an {explicit} bound on the error between the solution of a nonlinear stochastic differential equation and an easily computable linearisation approximation that has been used across many different applications \cite[e.g.]{Jazwinski_2014_StochasticProcessesFiltering, Sanz-AlonsoStuart_2017_GaussianApproximationsSmall,KaszasHaller_2020_UniversalUpperEstimate,ArchambeauEtAl_2007_GaussianProcessApproximations}, building upon previous studies \cite{Blagoveshchenskii_1962_DiffusionProcessesDepending,FreidlinWentzell_1998_RandomPerturbationsDynamical,Sanz-AlonsoStuart_2017_GaussianApproximationsSmall}.
The theory applies to fully non-autonomous SDEs with multiplicative noise and a random initial condition.
Our bound is written in terms of the scale of the ongoing noise (that is, a scaling of the diffusivity coefficient) and a measure of the uncertainty in the initial condition using the \(L_r\)-norm.
A comparison in \cref{sec:comparison} suggests that our bound on the moments is tighter than implied by the gold standard \cite{Sanz-AlonsoStuart_2017_GaussianApproximationsSmall} in the literature.

We also provided an explicit characteristation of the distribution of the solution to the linearised SDE, enabling efficient approximation of the original nonlinear SDE using solutions to the corresponding deterministic equation, whereas previously special cases of these computations were dispersed across other literature \cite{Jazwinski_2014_StochasticProcessesFiltering,Sanz-AlonsoStuart_2017_GaussianApproximationsSmall,SarkkaSolin_2019_AppliedStochasticDifferential}.
Finally, we validated the form of the bound numerically on three example stochastic differential equations, and in particular find that the strong error scales with the initial uncertainty and ongoing uncertainty exactly as predicted.

Our error bound may be useful to investigate the validity of SDE linearisations in application domains such as Stochastic Parameterisation \cite{BernerEtAl_2017_StochasticParameterizationNew,Palmer_2019_StochasticWeatherClimate,LeutbecherEtAl_2017_StochasticRepresentationsModel} or Data Assimilation \cite{BudhirajaEtAl_2019_AssimilatingDataModels,ReichCotter_2015_ProbabilisticForecastingBayesian,LawEtAl_2015_DataAssimilationMathematical,Harlim_2017_ModelErrorData}, or to improve such linearisations by monitoring the skew or kurtosis.

This paper also supplied theoretical and computational extension to the ``stochastic sensitivity'' tools introduced by \citet{Balasuriya_2020_StochasticSensitivityComputable}.
Stochastic sensitivity was hitherto derived as the variance of an unknown limiting distribution and could only be computed in two spatial dimensions; we established that stochastic sensitivity, in any number of dimensions, is computable as the operator norm of the covariance matrix of the linearised SDE.
We have also established that the limiting distribution in question is Gaussian, which may provide insight into properties of stochastic sensitivity as a means of uncertainty quantification in any model (not just in the fluids context) where an $ n $-dimensional state variable evolves according to a ``best available" model.

Our work on stochastic sensitivity computes the exact distribution of a Lagrangian coherent structure measure for the first time, and may be used to employ coherent structures as `data' in a Data Assimilation scheme. This idea is not new, and is explored in \cite{MacleanEtAl_2017_CoherentStructureApproach,MorzfeldEtAl_2018_FeaturebasedDataAssimilation,Schlueter-KuckDabiri_2019_ModelParameterEstimation}, albeit that the likelihood of a coherent structure was not computable. Our work fills that gap.
Moreover, most traditional LCS measures are completely deterministic measures, not accounting for any uncertainty in the driving velocity field, and the sensitivity of these methods to such uncertainty has not been investigated in detail.
The robustness of several LCS methods to stochastic noise has recently been explored in \cite{BadzaEtAl_2023_HowSensitiveAre}, but via stochastic simulation and summary statistics.
In this paper we have presented a theoretical result for characterising Lagrangian trajectory uncertainty, which can be used to perform a purely theoretical analysis of such sensitivity in LCS computations.


\appendix
\clearpage
\section{Preliminaries for proofs}\label{app:gauss}

There are several generic results and inequalities that we use several times throughout our proofs, which we state here for completeness.
We write \(W_t = \left(W_t^{(1)}, \hdots, W_t^{(m)}\right)^{\T}\) as the components of the canonical \(m\)-dimensional Wiener process, where each \(W_t^{(i)}\) are mutually independent 1-dimensional Wiener processes.
The flow map \(F_0^t: \R^n \to \R^n\) summarises solutions of the deterministic model \cref{eqn:ode_det}, given by
\begin{equation}
	F_{0}^{t}(x) = x + \int_{0}^{t}{u\left(F_{0}^{\tau}(x), \tau\right)\dif \tau},
	\label{eqn:flow_map_int}
\end{equation}
for an initial condition \(x \in \R^n\).
The spatial gradient (with respect to the initial condition) of the flow map solves the equation of variations associated with \cref{eqn:ode_det}, i.e.
\begin{equation}
	\dpd{}{t}\nabla F_0^t(x) = \nabla u\left(F_0^t(x), t\right)\nabla F_0^t(x).
	\label{eqn:eqn_of_vars}
\end{equation}

For any real numbers \(x_1,\hdots,x_p \geq 0\) and \(r \geq 1\),
\begin{equation}
	\left(\sum_{i=1}^p{x_i}\right)^r \leq p^{r-1}\sum_{i=1}^p{x_i^r}.
	\label{eqn:trinomial}
\end{equation}
This results from an application of the finite form of Jensen's inequality.
An implication of the equivalence of the \(L_1\) and Euclidean norms and \cref{eqn:trinomial} is that for any \(z \in \R^n\) and \(r \geq 1\),
\begin{equation}
	\norm{z}^r \leq \left(\sum_{i = 1}^n{\abs{z_i}}\right)^r \leq n^{r-1}\sum_{i=1}^n{\abs{z_i}^r},
	\label{eqn:norm_trinomial}
\end{equation}
where \(z_i\) denotes the \(i\)th component of \(z\).
If each component \(z_i\) of a vector \(z\) is bounded by a constant \(K\), then
\begin{equation}\label{eqn:bound_vector}
	\norm{z} \leq \sqrt{n} K.
\end{equation}
Similarly, if \(f: \R \to \R^n\) is a vector-valued function such that each component of \(f\) is integrable over an interval \([0,t]\), then for all \(r \geq 1\),
\begin{equation}
	\norm{\int_0^t{f\left(\tau\right)\dif\tau}}^r \leq t^{r-1}\int_0^t{\norm{f\left(\tau\right)}^r\dif\tau}.
	\label{eqn:convex_integral}
\end{equation}
This inequality results from an application of H\"{o}lder's inequality.

\section{Proof of \Cref{thm:main}}\label{app:main_thm_proof}
To prove the main result, we first require a lemma establishing a bound on the time integral of the expectation of the distance between the SDE solution and the reference deterministic trajectory.

\begin{lemma}\label{lem:z_int_bound}
	Let \(q \geq 1\) be such that \(\delta_q < \infty\), then for all \(\epsilon > 0\) and \(\tau \in [0,T]\)
	\begin{equation*}
		\avg{\int_0^t{\norm{y_\tau^{(\epsilon)} - F_0^t\!\left(x_0\right)}^q\dif\tau}} \leq H_1\!\left(q,t, K_{\nabla u}, K_{\sigma}\right)\epsilon^q + H_2\!\left(q,t, K_{\nabla u}\right)\delta_q^q,
	\end{equation*}
	where
	\begin{align*}
		H_1\!\left(q,t, K_{\nabla u}, K_{\sigma}\right) & \coloneqq 3^{q-1} n^{3q/2} K_{\sigma}^{q/2} G_{q/2} t^{q/2 + 1}\exp\left(3^{q-1} K_{\nabla u}^q t^q\right), \\
		H_2\!\left(q,t, K_{\nabla u}\right) & \coloneqq 3^{q-1} t \exp\left(3^{q-1} K_{\nabla u}^q t^q\right).
	\end{align*}
\end{lemma}

\begin{proof}
	Consider the integral form of \cref{eqn:sde_y},
	\[
		y_t^{(\epsilon)} = x + \int_0^t{u\left(y_\tau^{(\epsilon)}, \tau\right)\dif\tau} + \epsilon\int_0^t{\sigma\left(y_\tau^{(\epsilon)}, \tau\right)\dif W_\tau}.
	\]
	Using \cref{eqn:flow_map_int},
	\[
		y_t^{(\epsilon)} - F_0^t\!\left(x_0\right) = x - x_0 + \int_0^{t}{\left(u\left(y_\tau^{(\epsilon)}, \tau\right) - u\left(F_0^{\tau}\!\left(x_0\right), \tau\right)\right)\dif\tau} + \epsilon\int_0^t{\sigma\left(y_\tau^{(\epsilon)}, \tau\right)\dif W_\tau},
	\]
	and so
	\begin{equation}
		\avg{\norm{y_t^{(\epsilon)} - F_0^t\!\left(x_0\right)}^q} \leq \begin{multlined}[t]
			3^{q-1}\avg{\norm{x - x_0}^q} \\
			+ 3^{q-1}t^{q-1}\avg{\int_0^{t}{\norm{u\!\left(y_\tau^{(\epsilon)}, \tau\right) - u\left(F_0^{\tau}\!\left(x_0\right), \tau\right)}^q\dif\tau}} \\
			+ 3^{q-1}\epsilon^q\avg{\norm{\int_0^t{\sigma\!\left(y_\tau^{(\epsilon)}, \tau\right)\dif W_\tau}}^q},
		\end{multlined}
		\label{eqn:norm_y_t_tmp}
	\end{equation}
	using \cref{eqn:trinomial} followed by \cref{eqn:convex_integral}, and taking the expectation on both sides.

	Next, we establish a bound on the It\^o integral term in \cref{eqn:norm_y_t_tmp}.
	For \(i \in \set{1,\hdots,n}\), let \(\sigma_{i\cdot}\) denote the \(i\)th row of \(\sigma\).
	Define the stochastic process
	\[
		M_\tau^{(i)} \coloneqq \sigma_{i\cdot}\!\left(y_\tau^{(\epsilon)}, \tau\right)
	\]
	for \(\tau \in [0,t]\), so that
	\[
		\left[\int_0^t{\sigma\!\left(y_\tau^{(\epsilon)}, \tau\right)\dif W_\tau}\right]_i = \int_0^t{M_\tau^{(i)}\dif W_\tau}.
	\]
	Since \(y_t^{(\epsilon)}\) is a strong solution to \cref{eqn:sde_y}, we have that (e.g. see Definition 6.1.1 of \cite{KallianpurSundar_2014_StochasticAnalysisDiffusion})
	\[
		\int_0^t{\norm{M_\tau^{(i)}}^2\dif\tau} \leq \int_0^t{nK_\sigma^2\dif\tau} < \infty, \quad \text{almost surely},
	\]
	so we can apply the Burkholder-Davis-Gundy inequality (e.g. \cite[Thm. 5.6.3]{KallianpurSundar_2014_StochasticAnalysisDiffusion}) to \(M_\tau\), which asserts that there exists a constant \(G_{q/2} > 0\) depending only on \(q\) such that
	\begin{align*}
		\avg{\abs{\int_0^t{M_\tau^{(i)}\dif W_\tau}}^{q}} & \leq G_{q/2} \avg{\left(\int_{0}^t{\norm{\sigma_{i\cdot}\!\left(y_\tau^{(\epsilon)}, \tau\right)}^2\dif\tau}\right)^{q/2}} \\
		                                                  & \leq G_{q/2} n^p K_{\sigma}^{q/2} t^{q/2},
	\end{align*}
	where the second inequality uses \ref{hyp:sigma_bounds}.
	Then,
	\begin{equation}
		\avg{\norm{\int_0^t{\sigma\!\left(y_\tau^{(\epsilon)}, \tau\right)\dif W_\tau}}^{q}} \leq n^{3q/2} K_\sigma^{q/2} G_{q/2} t^{q/2},
		\label{eqn:z_sigma_bound}
	\end{equation}
	using \cref{eqn:norm_trinomial}.

	Applying the bound \cref{eqn:z_sigma_bound} to \cref{eqn:norm_y_t_tmp}, we have
	\begin{equation}\label{eqn:y_F_diff_inter}
		\avg{\norm{y_t^{(\epsilon)}\!-\!F_0^t\!\left(x_0\right)}^q} \leq \begin{multlined}[t]
			3^{q-1}\delta_q^q + 3^{q-1}\epsilon^q n^{3q/2}K_{\sigma}^{q/2} G_{q/2} t^{q/2} \\
			+ 3^{q-1}t^{q-1}\avg{\int_0^{t}{\norm{u\!\left(y_\tau^{(\epsilon)}, \tau\right) - u\!\left(F_0^{\tau}\!\left(x_0\right), \tau\right)}}^q\dif\tau}.
		\end{multlined}
	\end{equation}
	We note that \(\avg{\norm{y_t^{(\epsilon)} - F_0^t\!\left(x\right)}^q} < \infty\) from \ref{hyp:u_bounds}, so by Tonelli's theorem (e.g. \cite[Thm. 2.3.9]{Bremaud_2020_ProbabilityTheoryStochastic}),
	\begin{equation*}
		\avg{\int_0^{t}{\norm{y_\tau^{(\epsilon)} - F_0^\tau\!\left(x_0\right)}^q\dif\tau}} = \int_0^{t}{\avg{\norm{y_\tau^{(\epsilon)} - F_0^\tau\!\left(x_0\right)}^q}\dif\tau}.
		\label{eqn:y_tonelli}
	\end{equation*}
	Now, using the Lipschitz continuity of \(u \) from \ref{hyp:sigma_deriv_bound} on \cref{eqn:y_F_diff_inter} and interchanging the expectation and integral,
	\[
		\avg{\norm{y_t^{(\epsilon)} - F_0^t\!\left(x_0\right)}^q} \leq \begin{multlined}[t]
			3^{q-1}\delta_q^q + 3^{q-1} K_{\nabla u}^q t^{q-1} \int_0^{t}{\avg{\norm{y_\tau^{(\epsilon)} - F_0^{\tau}\!\left(x_0\right)}^q}\dif\tau} \\
			+ 3^{q-1}\epsilon^q n^{3q/2}K_{\sigma}^{q/2} G_{q/2} t^{q/2}.
		\end{multlined}
	\]
	Applying Gr\"{o}nwall's inequality and using the monotonicity of the resulting bound in \(t\), we have that for any \(\tau \in [0,t]\),
	\[
		\avg{\norm{y_\tau^{(\epsilon)} - F_0^\tau\!\left(x_0\right)}^q}  \leq \begin{multlined}[t]
			3^{q-1}\epsilon^q n^{3q/2}S^{q/2} G_{q/2} t^{q/2}\exp\left(3^{q-1} K_{\nabla u}^q t^q\right) \\
			+ 3^{q-1}\exp\!\left(3^{q-1} K_{\nabla u}^q t^q\right)\delta_q^q .
		\end{multlined}
	\]
	Integrating both sides with respect to time
	and again using Tonelli's theorem, we have
	\begin{align*}
		\avg{\int_0^t{\norm{y_\tau^{(\epsilon)} - F_0^\tau\!\left(x_0\right)}^q\dif\tau}} \leq \begin{multlined}[t]
			                                                                                       3^{q-1}n^{3q/2}S^{q/2} G_{q/2} t^{q/2 + 1}\exp\!\left(3^{q-1} K_{\nabla u}^q t^q\right)\epsilon^q  \\
			                                                                                       + 3^{q-1}t\exp\left(3^{q-1} K_{\nabla u}^q t^q\right)\delta_q^q ,
		                                                                                       \end{multlined}
	\end{align*}
	as desired.
\end{proof}

With these bounds established, we can now prove \Cref{thm:main}.
Subtracting the integral forms of \cref{eqn:linear_sde_inform} and \cref{eqn:sde_y} gives
\begin{align*}
	y_t^{(\epsilon)} - l_t^{(\epsilon)} & = \begin{multlined}[t]
		                                        \int_0^t{\left[u\!\left(y_\tau^{(\epsilon)}, \tau\right) - u\!\left(F_0^\tau\!\left(x_0\right), \tau\right) - \nabla u\!\left(F_0^\tau\!\left(x_0\right)\right)\left(l_\tau^{(\epsilon)} - F_0^\tau\!\left(x_0\right)\right)\right]\dif\tau} \\
		                                        + \int_0^t{\left[\epsilon\sigma\!\left(y_\tau^{(\epsilon)}, \tau\right) - \epsilon\sigma\!\left(F_0^{\tau}\!\left(x_0\right), \tau\right)\right]\dif W_\tau}
	                                        \end{multlined}              \\
	                                    & = \begin{multlined}[t]
		                                        \int_0^t{\left[u\!\left(y_\tau^{(\epsilon)}, \tau\right) - \left(u\!\left(F_0^\tau\!\left(x_0\right), \tau\right) + \nabla u\!\left(F_0^\tau\!\left(x_0\right)\right)\left(y_\tau^{(\epsilon)} - F_0^\tau\!\left(x_0\right)\right)\right)\right]\dif\tau} \\
		                                        + \int_0^t{\nabla u\!\left(F_0^\tau\!\left(x_0\right), \tau\right)\left[y_\tau^{(\epsilon)} - l_\tau^{(\epsilon)}\right]\dif \tau} \\
		                                        + \epsilon\int_0^t{ \left[\sigma\!\left(y_\tau^{(\epsilon)}, \tau\right) - \sigma\!\left(F_0^{\tau}\!\left(x_0\right), \tau\right)\right]\dif W_\tau}
	                                        \end{multlined} \\
	                                    & = A(t) + B(t) + \epsilon C(t),
\end{align*}
where
\begin{align*}
	A(t) & \coloneqq \int_0^t{\left[u\!\left(y_\tau^{(\epsilon)}, \tau\right) - \left(u\!\left(F_0^\tau\!\left(x_0\right), \tau\right) + \nabla u\!\left(F_0^\tau\!\left(x_0\right)\right)\left(y_\tau^{(\epsilon)} - F_0^\tau\!\left(x_0\right)\right)\right)\right]\dif\tau} \\
	B(t) & \coloneqq \int_0^t{\nabla u\!\left(F_0^\tau\!\left(x_0\right), \tau\right)\left[y_\tau^{(\epsilon)} - l_\tau^{(\epsilon)}\right]\dif \tau}                                                                                                                          \\
	C(t) & \coloneqq \int_0^t{\left[\sigma\!\left(y_\tau^{(\epsilon)}, \tau\right) - \sigma\!\left(F_0^{\tau}\!\left(x_0\right), \tau\right)\right]\dif W_\tau}.
\end{align*}
Then, using \cref{eqn:trinomial} and taking expectation,
\begin{equation}
	\avg{\norm{y_t^{(\epsilon)} - l_t^{(\epsilon)}}^r} \leq 3^{r-1}\left(\avg{\norm{A(t)}^r} + \avg{\norm{B(t)}^r} + \epsilon^r\avg{\norm{C(t)}^r}\right) .
	\label{eqn:diff_decomp}
\end{equation}
First consider \(A(t)\), for which the integrand is the expected difference between the drift \(u\) evaluated along SDE solution and the first-order Taylor expansion of \(u\) about the reference deterministic trajectory.
Since for any \(t \in [0,T]\), \(u\left(\cdot, t\right)\) is twice continuously differentiable under \ref{hyp:coef_cont}, for each \(i = 1,\hdots,n\) there exists by Taylor's theorem (e.g. see \cite[Cor. A9.3.]{HubbardHubbard_2009_VectorCalculusLinear}) a function \(R_i: \R^n \times [0,T] \to \R\) such that
\begin{equation}
	u_i\!\left(z, \tau\right) = u_i\!\left(F_0^\tau\!\left(x_0\right), \tau\right) + \left[\nabla u_i\left(F_0^\tau\!\left(x_0\right), \tau\right)\right]\left(z - F_0^\tau\!\left(x_0\right)\right) + R_i\!\left(z, \tau\right)
	\label{eqn:taylor_expan}
\end{equation}
for any \(z \in \R^n\), where \(u_i\) denotes the \(i\)th component of \(u\).
The function \(R_i\) satisfies
\begin{equation}
	\abs{R_i(z, \tau)} \leq \frac12\norm{\nabla\nabla u_i\left(F_0^\tau(x), t\right)}\norm{z - F_0^\tau\!\left(x_0\right)}^2 \leq \frac{K_{\nabla\nabla u}}{2}\norm{z - F_0^\tau\!\left(x_0\right)}^2.
	\label{eqn:rem_ineq}
\end{equation}
Let \(R\!\left(z, \tau\right) \coloneqq \left( R_1\!\left(z, \tau\right), \hdots, R_n\!\left(z, \tau\right)\right)^{\T}\), then
\[
	A(t) = \int_0^t{R\left(y_{t}^{(\epsilon)}, \tau\right)\dif\tau},
\]
and since each component of \(R\) is bounded as in \cref{eqn:rem_ineq}, using \cref{eqn:bound_vector}
\[
	\norm{R\!\left(y_t^{(\epsilon)}, \tau\right)} \leq \frac{\sqrt{n} K_{\nabla\nabla u}}{2}\norm{y_t^{(\epsilon)} - F_0^\tau\!\left(x_0\right)}^2.
\]
Taking the norm and expectation then gives
\begin{align*}
	\avg{\norm{A(t)}^r} & = \avg{\norm{\int_0^t{R\!\left(y_t^{(\epsilon)}, \tau\right)\dif\tau}}^r}                                                                                                                  \\
	                    & \leq t^{r-1}\avg{\int_0^t{\norm{R\!\left(y_\tau^{(\epsilon)}, \tau\right)}^r\dif\tau}}                                                                                                     \\
	                    & \leq \frac{t^{r-1}n^{r/2}K_{\nabla\nabla u}^r}{2^r}\avg{\int_0^t{\norm{y_\tau^{(\epsilon)} - F_0^\tau\!\left(x_0\right)}^{2r}\dif\tau}}                                                    \\
	                    & \leq \begin{multlined}[t]
                     \frac{t^{r-1} n^{r/2} K^r_{\nabla\nabla u}H_1\!\left(2r,t, K_{\nabla u}, K_{\sigma}\right)}{2^r}\epsilon^{2r} \\ 
                     + \frac{t^{r-1} n^{r/2} K_{\nabla\nabla u}^r H_2\!\left(2r,t, K_{\nabla u}\right)}{2^{r}}\delta_{2r}^{2r}, 
                            \end{multlined}\numberthis\label{eqn:A_ineq}
\end{align*}
where the first inequality uses \cref{eqn:convex_integral}, and \(H_1\) and \(H_2\) are obtained from \Cref{lem:z_int_bound}.

Next, consider \(B(t)\), for which
\begin{equation}\label{eqn:B_ineq}
	\avg{\norm{B(t)}^r} \leq \int_0^t{t^{r-1}K_{\nabla u}^r\avg{\norm{y_\tau^{(\epsilon)} - l_t^{(\epsilon)}}^r}\dif\tau}.
\end{equation}
using \cref{eqn:convex_integral} and then \ref{hyp:u_bounds}, and interchanging the expectation and the integral uses the fact that that \(\avg{\norm{y_\tau^{(\epsilon)}}} < \infty\) and \(\avg{\norm{l_\tau^{(\epsilon)}}} < \infty\).

Finally, consider \(C(t)\).
For each \(i \in \set{1,\hdots, n}\), define the stochastic process
\[
	N_\tau^{(i)} \coloneqq \sigma_{i\cdot}\left(y_\tau^{(\epsilon)},\tau\right) - \sigma_{i\cdot} \left(F_0^\tau\!\left(x_0\right), \tau\right).
\]
Then, the \(i\)th component of \(C(t)\) is
\[
	\left[C(t)\right]_i = \int_0^t{N_\tau^{(i)}\dif W_\tau}.
\]
From \ref{hyp:sigma_bounds} and using \cref{eqn:bound_vector},
\[
	\int_0^t{\norm{N_\tau^{(i)}}^2\dif\tau} \leq \int_0^t{4nK_\sigma^2\dif\tau} < \infty,
\]
so we can apply the Burkholder-Davis-Gundy inequality on \(N_\tau^{(i)}\) to write
\begin{align*}
	\avg{\abs{\left[C(t)\right]_i}^{r}} & \leq G_{r/2}\avg{\left(\int_{0}^t{\norm{\sigma_{i\cdot}\left(y_\tau^{(\epsilon)}, \tau\right) - \sigma_{i\cdot} \left(F_0^\tau\!\left(x_0\right), \tau\right)}^2\dif\tau}\right)^{r/2}} \\
	                                    & \leq G_{r/2}\avg{\left(\int_0^{t}{ K_{\nabla\sigma}^2\norm{y_\tau^{(\epsilon)} - F_0^\tau\!\left(x_0\right)}^2\dif\tau}\right)^{r/2}}                                                   \\
	                                    & \leq G_{r/2} K_{\nabla \sigma}^r t^{r/2 - 1}\avg{\int_0^t{\norm{y_\tau^{(\epsilon)} - F_0^\tau\!\left(x_0\right)}^r\dif\tau}}                                                           \\
	                                    & \leq \begin{multlined}[t]
		                                           G_{r/2} K_{\nabla\sigma}^r t^{r/2 - 1} H_1\!\left(r,t, K_{\nabla u}, K_{\sigma}\right)\epsilon^r \\ 
                                             + G_{r/2} K_{\nabla\sigma}^r t^{r/2 - 1} H_2\!\left(r,t, K_{\nabla u}\right)\delta_r^r,
                                           \end{multlined} \numberthis\label{eqn:N_comp_ineq}
\end{align*}
where the second inequality uses the Lipschitz condition on \(\sigma\) in \ref{hyp:sigma_deriv_bound}, the third inequality uses \cref{eqn:convex_integral}, and the fourth inequality uses \Cref{lem:z_int_bound} with \(q = r\).
Then, we have
\begin{equation}\label{eqn:C_ineq}
	\avg{\norm{C(t)}^r} \leq \begin{multlined}[t] 
        n^{r}G_{r/2} K_{\nabla\sigma}^r t^{r/2 - 1} H_1\!\left(r,t, K_{\nabla u}, K_{\sigma}\right) \epsilon^r \\ 
         + n^r G_{r/2} K_{\nabla\sigma}^r t^{r/2 - 1} H_2\!\left(r,t, K_{\nabla u}\right)\delta_r^r,
         \end{multlined}
\end{equation}
using \cref{eqn:norm_trinomial}, and then \cref{eqn:N_comp_ineq}.

Combining \cref{eqn:A_ineq}, \cref{eqn:B_ineq} and \cref{eqn:C_ineq} into \cref{eqn:diff_decomp}, we have
\[
	\avg{\norm{y_t^{(\epsilon)} - l_t^{(\epsilon)}}^r} \leq \begin{multlined}[t]
		\frac{3^{r-1} t^{r-1}n^{r/2}K_{\nabla\nabla u}^r H_1\!\left(2r,t, K_{\nabla u}, K_{\sigma}\right)}{2^r}\epsilon^{2r} \\
		+ \frac{3^{r-1} t^{r-1} n^{r/2} K_{\nabla\nabla u}^r H_2\!\left(2r,t, K_{\nabla u}\right)}{2^r}\delta_{2r}^{2r} \\
		+ \int_0^t{3^{r-1}t^{r-1}K_{\nabla u}^r\avg{\norm{y_t^{(\epsilon)} - l_t^{(\epsilon)}}^r}\dif\tau} \\
		+ 3^{r-1}G_{r/2} K_{\nabla\sigma}^r t^{r/2 - 1} H_1\!\left(r,t, K_{\nabla u}, K_{\sigma}\right)\epsilon^{2r} \\
		+ 3^{r-1} G_{r/2} K_{\nabla\sigma}^r t^{r/2 - 1} H_2\!\left(r,t, K_{\nabla u}\right)\epsilon^r\delta_r^r.
	\end{multlined}
\]
Applying Gr\"{o}nwall's inequality, noting that \(H_1\) and \(H_2\) are non-decreasing in \(t\), we have
\[
	\avg{\norm{y_t^{(\epsilon)} - l_t^{(\epsilon)}}^r} \leq \begin{multlined}[t]
		\frac{3^{r-1} t^{r-1}n^{r/2}K_{\nabla\nabla u}^r H_1\!\left(2r,t, K_{\nabla u}, K_{\sigma}\right)}{2^r}\exp\left(3^{r-1}t^r K_{\nabla u}^r\right)\epsilon^{2r} \\
		+ \frac{3^{r-1} t^{r-1} n^{r/2} K_{\nabla\nabla u}^r H_2\!\left(2r,t, K_{\nabla u}\right)}{2^r}\exp\left(3^{r-1}t^r K_{\nabla u}^r\right)\delta_{2r}^{2r} \\
		+ 3^{r-1}G_{r/2} K_{\nabla\sigma}^r t^{r/2 - 1}\exp\left(3^{r-1}t^r K_{\nabla u}^r\right) H_1\!\left(r,t, K_{\nabla u}, K_{\sigma}\right)\epsilon^{2r} \\
		+ 3^{r-1} G_{r/2} K_{\nabla\sigma}^r t^{r/2 - 1}\exp\left(3^{r-1}t^r K_{\nabla u}^r\right) H_2\!\left(r,t, K_{\nabla u}\right)\epsilon^r\delta_r^r.
	\end{multlined}
\]
Set
\begin{subequations}\label{eqn:bound_defns}
	\begin{align}
		D_1\!\left(r,t, K_{\nabla u}, K_\sigma\right) & \coloneqq 3^{r-1}\exp\left(3^{r-1} t^r K_{\nabla u}^r\right)K_{M}\!\left(r, t, K_{\nabla u}, K_\sigma\right) \\
		D_2\!\left(r,t, K_{\nabla u}\right) & \coloneqq 3^{r-1}t^{r-1} n^{r/2}H_{2}\!\left(2r,t, K_{\nabla u}\right)\exp\left(3^{r-1}t^r K_{\nabla u}^r\right)                                                        \\
		D_3\!\left(r,t, K_{\nabla u}\right) & \coloneqq 3^{r-1}G_{r/2} t^{r/2 - 1} H_2\!\left(r,t, K_{\nabla u}\right)\exp\left(3^{r-1}t^r K_{\nabla u}^r\right),
	\end{align}
\end{subequations}
where 
\[
K_{M}\!\left(r, t, K_{\nabla u}, K_\sigma\right) \coloneqq \max\set{\frac{t^{r-1}n^{r/2}H_{1}\!\left(2r,t, K_{\nabla u}, K_{\sigma}\right)}{2^r},\, G_{r/2} t^{r/2 - 1} H_1\!\left(r,t, K_{\nabla u}, K_{\sigma}\right)},
\]
then we have shown the desired result.

\section{Proof of \Cref{thm:limit_sol}}\label{app:limit_sol_proof}

Next, we show that the strong solution to the linearised SDE \cref{eqn:linear_sde_inform} can be written as the independent sum \cref{eqn:linear_sol}.
Let
\[
	M_t = h\!\left(l_t^{(\epsilon)}, t\right) \coloneqq \frac{1}{\epsilon}\left[\nabla F_0^t\!\left(x_0\right)\right]^{-1}\left(l_t^{(\epsilon)} - F_0^t\!\left(x_0\right)\right),
\]
where \(l_t^{(\epsilon)}\) is the strong solution to \cref{eqn:linear_sde_inform}.
Then, the required derivatives for applying It\^o's Lemma (e.g. see Theorem 5.5.1 of \cite{KallianpurSundar_2014_StochasticAnalysisDiffusion}) are
\begin{align*}
	M_0                                              & = h\!\left(l_0^{(\epsilon)}, 0\right) = \frac{1}{\epsilon}\left(x - x_0\right)                                                                                                                                       \\
	\dpd{h}{t}                                       & = \begin{multlined}[t]
		                                                     -\frac{1}{\epsilon}\left[\nabla F_0^t\!\left(x_0\right)\right]^{-1} \dpd{\nabla F_0^t\!\left(x_0\right)}{t}\left[\nabla F_0^t\!\left(x_0\right)\right]^{-1}\left(l_t^{(\epsilon)} - F_0^t\!\left(x_0\right)\right) \\
		                                                     - \frac{1}{\epsilon}\left[\nabla F_0^t\!\left(x_0\right)\right]^{-1}u\!\left(F_0^t\!\left(x_0\right), t\right)
	                                                     \end{multlined} \\
	\nabla h\!\left(l_t^{(\epsilon)}, t\right)       & = \frac{1}{\epsilon}\left[\nabla F_0^t\!\left(x_0\right)\right]^{-1}                                                                                                                                                 \\
	\nabla\nabla h\!\left(l_t^{(\epsilon)}, t\right) & = O,
\end{align*}
where in computing the \(t\) derivative we have used the fact that \(F_0^t\!\left(x_0\right)\) solves the deterministic ODE \cref{eqn:ode_det}.
Thus,
\begin{align*}
	M_t & = \begin{multlined}[t]
		        \frac{1}{\epsilon}\left(x - x_0\right) + \frac{1}{\epsilon}\int_0^t\left(-\left[\nabla F_0^\tau\!\left(x_0\right)\right]^{-1} \dpd{\nabla F_0^\tau\!\left(x_0\right)}{\tau}\left[\nabla F_0^\tau\!\left(x_0\right)\right]^{-1}\left(l_\tau^{(\epsilon)} - F_0^\tau\!\left(x_0\right) \right)\right. \\
		        \left. - \left[\nabla F_0^\tau\!\left(x_0\right)\right]^{-1}u\!\left(F_0^\tau\!\left(x_0\right), \tau\right)\right. \\
		        \left. + \left[\nabla F_0^\tau\!\left(x_0\right)\right]^{-1}\left[ u\!\left(F_0^\tau\!\left(x_0\right), \tau\right) + \nabla u\!\left(F_0^\tau\!\left(x_0\right), \tau\right) \left(l_\tau^{(\epsilon)} - F_0^\tau\!\left(x_0\right)\right)\right] \right)\dif\tau \\
		        + \int_0^t{\left[\nabla F_0^\tau\!\left(x_0\right)\right]^{-1}\sigma\!\left(F_0^\tau\!\left(x_0\right), \tau\right)\dif W_\tau}
	        \end{multlined} \\
	    & = \begin{multlined}[t]
		        \frac{1}{\epsilon}\left(x - x_0\right) + \frac{1}{\epsilon}\int_0^t\left(-\left[\nabla F_0^\tau\!\left(x_0\right)\right]^{-1} \nabla u\!\left(F_0^\tau\!\left(x_0\right), \tau\right)\left(l_\tau^{(\epsilon)} - F_0^\tau\!\left(x_0\right) \right)\right. \\
		        \left. + \left[\nabla F_0^\tau\!\left(x_0\right)\right]^{-1}\nabla u\!\left(F_0^\tau\!\left(x_0\right), \tau\right) \left(l_\tau^{(\epsilon)} - F_0^\tau\!\left(x_0\right)\right) \right)\dif\tau \\
		        + \int_0^t{\left[\nabla F_0^\tau\!\left(x_0\right)\right]^{-1}\sigma\!\left(F_0^\tau\!\left(x_0\right), \tau\right)\dif W_\tau}
	        \end{multlined}              \\
	    & = \frac{1}{\epsilon}\left(x - x_0\right) + \int_0^t{\left[\nabla F_0^\tau\!\left(x_0\right)\right]^{-1}\sigma\!\left(F_0^\tau\!\left(x_0\right), \tau\right)\dif W_\tau},
\end{align*}
where we reach the second line by using the equation of variations \cref{eqn:eqn_of_vars} satisfied by \(\nabla F_0^\tau\!\left(x_0\right)\).
It follows that
\[
	l_t^{(\epsilon)} = \nabla F_0^t\!\left(x_0\right)\left(x - x_0\right) +  F_0^t\!\left(x_0\right) + \epsilon\int_0^t{\nabla F_0^t\!\left(x_0\right)\left[\nabla F_0^\tau\!\left(x_0\right)\right]^{-1} \sigma\!\left(F_0^\tau\!\left(x_0\right), \tau\right)\dif \tau}
\]
is a strong solution to \cref{eqn:linear_sde_inform}.\ to \cref{eqn:linear_sol}.

By \ref{hyp:init_indep}, \(x\) is independent of the Wiener process \(W_t\), and since independence is preserved under limits and linear transformations it follows that the It\^o integral \(\epsilon \nabla F_0^t\!\left(x_0\right)\int_0^t{L\!\left(x_0, \tau\right)\dif W_\tau}\) and \(\nabla F_0^t\!\left(x_0\right)\left(x - x_0\right)\) are independent.

\section{Proof of \Cref{cor:limit_moments}}\label{app:limit_moments_proof}
We first establish that the It\^o integral of a matrix-valued deterministic function with respect to a multidimensional Wiener process is a multidimensional Gaussian process.
This is a well-known result in the scalar case, and the extension to our case is straightforward.
\begin{lemma}\label{lem:det_gauss}
	Let \(a,b \in \R\) and let \(g: [a,b] \to \R^{n\times n}\) be a matrix-valued deterministic function such that each element of \(g\) is It\^o-integrable.
	Consider the It\^o integral
	\[
		\mathcal{I}[g] \coloneqq \int_{a}^b{g(t)\dif W_t},
	\]
	Then, the integral \(\mathcal{I}[g]\) is a \(n\)-dimensional multivariate Gaussian random variable.
\end{lemma}
\begin{proof}
	For \(i,j \in \set{1,\hdots,n}\), let \(g_{ij}: [a,b] \to \R\) be the \((i,j)\)th element of \(g\).
	Then, let
	\[
		\mathcal{I}[g_{ij}] \coloneqq \int_a^b{g_{ij}(t)\dif W_t^{(i)}},
	\]
	so that the \(i\)th element of \(\mathcal{I}[g]\) is
	\[
		\mathcal{I}[g]_i = \sum_{j = 1}^n{\mathcal{I}\left[g_{ij}\right]}.
	\]
	Each \(\mathcal{I}[g_{ij}]\) is an It\^o integral of a deterministic, scalar-valued function with respect to a one-dimensional Brownian motion, which is well-known to be a Gaussian process (e.g. see \cite[Lem. 4.3.11]{Applebaum_2004_LevyProcessesStochastic}).
	Moreover, each element of \(\mathcal{I}[g]\) is the sum of independent Gaussian random variables and is therefore itself Gaussian.
	Hence, \(\mathcal{I}[g]\) follows a multivariate Gaussian distribution.
\end{proof}

Now, we move onto showing \Cref{cor:limit_moments}.
Consider the It\^o integral
\[
	\mathcal{I}[L] = \int_0^t{L\!\left(x_0, \tau\right)\dif W_\tau}.
\]
For any fixed \(t \in [0,T]\), the integrand is a deterministic, matrix-valued function, and is therefore follows a \(n\)-dimensional Gaussian distribution.
Moreover \cite{KallianpurSundar_2014_StochasticAnalysisDiffusion},
\[
	\avg{\mathcal{I}[L]} = 0,
\]
and
\[
	\var{\mathcal{I}[L]} = \avg{\left(\int_0^t{L\!\left(x_0, \tau\right)\dif W_\tau}\right)\left(\int_0^t{L\!\left(x_0, \tau\right)\dif W_\tau}\right)^{\T}}.
\]
Let \(L_{ij}\) denote the \((i,j)\)th element of \(L\), then the \((i,j)\)th element of the variance is
\begin{align*}
	\left[\var{\mathcal{I}[L]}\right]_{ij} 
	 & = \sum_{k=1}^{m}\sum_{l=1}^m\avg{\left(\int_0^t{L_{ik}\!\left(x_0, \tau\right)\dif W_{\tau}^{(k)}}\right)\left(\int_0^t{L_{jl}\!\left(x_0, \tau\right)\dif W_{\tau}^{(l)}}\right)} \\
	 & = \begin{multlined}[t]
		     \sum_{k=1}^{m}\avg{\left(\int_0^t{L_{ik}\!\left(x_0, \tau\right)\dif W_{\tau}^{(k)}}\right)\!\left(\int_0^t{L_{jk}\!\left(x_0, \tau\right)\dif W_{\tau}^{(k)}}\right)} \\
		     \!\!+ \sum_{k=1}^{m}\sum_{\substack{l=1 \\ l \neq k}}^m\avg{\left(\int_0^t{L_{ik}\!\left(x_0, \tau\right)\dif W_{\tau}^{(k)}}\right)}\!\avg{\left(\int_0^t{L_{jl}\!\left(x_0, \tau\right)\dif W_{\tau}^{(l)}}\right)}
	     \end{multlined}                                                                        \\
	 & = \sum_{k=1}^m{\int_0^t{L_{ik}\!\left(x_0, \tau\right)L_{jk}\!\left(x_0, \tau\right)\dif\tau}},
\end{align*}
where the second equality  the fact that \(W_t^{(k)}\) is independent of \(W_t^{(l)}\) for \(k \neq l\) and the third equality uses It\^o's isometry \cite{KallianpurSundar_2014_StochasticAnalysisDiffusion}.
Hence, we have that
\[
	\int_0^t{L\!\left(x_0, \tau\right)\dif\tau} \isGauss{0, \int_0^t{L\!\left(x_0, \tau\right)L\!\left(x_0, \tau\right)^{\T}}\dif\tau},
\]
completing the proof of \Cref{thm:limit_sol}.
Next, we show that the mean and covariance of \(l_t^{(\epsilon)}\) are given explicitly by \cref{eqn:mean_expl_eqn} and \cref{eqn:pi_expl_eqn} respectively.
It follows immediately from \cref{eqn:linear_sol} that the mean of \(l_t^{(\epsilon)}\) is
\begin{align*}
	\avg{l_t^{(\epsilon)}} & = \avg{\nabla F_0^t\!\left(x_0\right) \left(x - x_0\right)} + F_0^t\!\left(x_0\right) + \epsilon \nabla F_0^t\!\left(x_0\right)\avg{\int_0^t{L\left(x_0,\tau\right)\dif\tau}} \\
	                       & = \nabla F_0^t(x) \left(\avg{x} - x_0\right) + F_0^t\!\left(x_0\right),
\end{align*}
thus showing \cref{eqn:mean_expl_eqn}.

Since the two summands in \cref{eqn:linear_sol} are independent, the variance of \(l_t^{(\epsilon)}\) is
\begin{align*}
	\var{l_t^{(\epsilon)}} & = \var{\nabla F_0^t\!\left(x_0\right) \left(x - x_0\right)} + \var{\epsilon \nabla F_0^t\!\left(x_0\right)\int_0^t{L\!\left(x_0,\tau\right)\dif W_\tau}}                                      \\
	                       & = \nabla F_0^t\!\left(x_0\right) \left(\var{x} + \epsilon^2\int_0^t{L\!\left(x_0, \tau\right)L\!\left(x_0, \tau\right)^{\T} \dif\tau}\right) \left[\nabla F_0^t\!\left(x_0\right)\right]^{\T}
\end{align*}
where the variance of the It\^o integral was established in \Cref{app:limit_sol_proof}.

Finally, we show \Cref{rem:cov_ode}, i.e. that \(\var{l_t^{(\epsilon)}}\) is the solution to the matrix differential equation \cref{eqn:pi_ode}.
Directly differentiating the expression \cref{eqn:pi_expl_eqn}
\begin{align*}
	\dod{\var{l_t^{(\epsilon)}}}{t} & = \begin{multlined}[t]
		                                    \dpd{\nabla F_0^t\!\left(x_0\right)}{t}\left(\var{x} + \epsilon^2 \int_0^t{L\!\left(x_0, \tau\right)L\!\left(x_0, \tau\right)^{\T}\dif\tau} \right)\left[\nabla F_0^t\!\left(x_0\right)\right]^{\T} \\
		                                    + \nabla F_0^t\!\left(x_0\right)\left(\var{x} + \epsilon^2 \int_0^t{L\!\left(x_0, \tau\right)L\!\left(x_0, \tau\right)^{\T}\dif\tau} \right)\left[\dpd{\nabla F_0^t\!\left(x_0\right)}{t}\right]^{\T} \\
		                                    + \epsilon^2\nabla F_0^t\!\left(x_0\right)L\!\left(x_0, t\right) L\!\left(x_0, t\right)^{\T}\left[\nabla F_0^t\!\left(x_0\right)\right]^{\T}
	                                    \end{multlined}                                                             \\
	                                & = \begin{multlined}[t]
		                                    \nabla u\!\left(F_0^t\!\left(x_0\right), t\right)\nabla F_0^t\!\left(x_0\right)\left(\var{x} + \epsilon^2 \int_0^t{L\!\left(x_0, \tau\right)L\!\left(x_0, \tau\right)^{\T}\dif\tau} \right)\left[\nabla F_0^t\!\left(x_0\right)\right]^{\T} \\
		                                    + \nabla F_0^t\!\left(x_0\right)\left(\var{x} + \epsilon^2 \int_0^t{L\!\left(x_0, \tau\right)L\!\left(x_0, \tau\right)^{\T}\dif\tau} \right)\left[\nabla F_0^t\!\left(x_0\right)\right]^{\T}\left[\nabla u\!\left(F_0^t\!\left(x_0\right), t\right)\right]^{\T} \\
		                                    + \epsilon^2\sigma\!\left(F_0^t\!\left(x_0\right), t\right)\sigma\!\left(F_0^t\!\left(x_0\right), t\right)^{\T}
	                                    \end{multlined} \\
	                                & = \begin{multlined}[t]
		                                    \nabla u\!\left(F_0^t\!\left(x_0\right), t\right) \var{l_t^{(\epsilon)}} + \var{l_t^{(\epsilon)}}\left[\nabla u\!\left(F_0^t\!\left(x_0\right), t\right)\right]^{\T} \\
		                                    + \epsilon^2\sigma\!\left(F_0^t\!\left(x_0\right), t\right)\sigma\!\left(F_0^t\!\left(x_0\right), t\right)^{\T},
	                                    \end{multlined}
\end{align*}
where the second inequality has used the equation of variations \cref{eqn:eqn_of_vars}.

\section{Proof of \Cref{thm:s2_calculation}}\label{app:s2_calculation_proof}
Let \(t \in [0,T]\) and consider the solutions \(y_t^{(\epsilon)}\) to \cref{eqn:sde_y} and \(l_t^{(\epsilon)}\) to \cref{eqn:linear_sde_inform} subject to the fixed initial condition \(x_0 \in \R^n\).
On the vector space of \(n\)-dimensional random vectors with each component having finite expectation and variance, define the function \(\rho\) as
\[
	\rho\!\left(z\right) \coloneqq \norm{\var{z}}^{\frac12}.
\]
Then, \(\rho\) is a semi-norm, which can be verified using properties of the spectral norm and the Cauchy-Schwarz inequality.
This proof is provided in the supplementary materials.
Then,
\begin{align*}
	\abs{\!\norm{\var{y_t^{(\epsilon)}\!}}^{\frac12}\!\!\! - \norm{\var{l_t^{(\epsilon)}\!}}^{\frac12}\!}
	 & = \abs{\rho\!\left(y_t^{(\epsilon)}\right) - \rho\!\left(l_t^{(\epsilon)}\right) }                                                                                                                                   \\
	 & \leq \rho\!\left(y_t^{(\epsilon)} - l_t^{(\epsilon)}\right)                                                                                                                                                          \\
	 & = \norm{\avg{\!\left(y_t^{(\epsilon)}\!- l_t^{(\epsilon)}\right)\!\left(y_t^{(\epsilon)}\!- l_t^{(\epsilon)}\right)^{\T}}\! - \!\avg{y_t^{(\epsilon)}\!- l_t^{(\epsilon)}}\!\avg{y_t^{(\epsilon)}\!- l_t^{(\epsilon)}}^{\T}\!}^{\frac12} \\
	 & \leq \left(\avg{\norm{y_t^{(\epsilon)} - l_t^{(\epsilon)}}^{2}} + \avg{\norm{y_t^{(\epsilon)} - l_t^{(\epsilon)}}}^2\right)^{\frac12}                                                                                    \\
	 & \leq \left(\left(K_{\nabla\nabla u} + K_{\nabla\sigma}\right)D_1(2,t)+ \left(K_{\nabla\nabla u} + K_{\nabla\sigma}\right)^2 D_1(1,t)^2\right)^{1/2}\epsilon^2
\end{align*}
where the first inequality uses the reverse triangle inequality, the second inequality uses the Jensen's inequality and properties of the spectral norm, and the third inequality results from \Cref{thm:main}.
Thus,
\[
	\abs{\norm{\frac{1}{\epsilon^2}\var{y_t^{(\epsilon)}}}^{\frac12} - \norm{\frac{1}{\epsilon^2}\var{l_t^{(\epsilon)}}\!}^{\frac{1}{2}}} \leq \begin{multlined}[t]
 \left(\left(K_{\nabla\nabla u} + K_{\nabla\sigma}\right)D_1(2,t) \right. \\
     \left. + \left(K_{\nabla\nabla u} + K_{\nabla\sigma}\right)^2 D_1(1,t)^2\right)^{\frac{1}{2}}\epsilon,
\end{multlined}
\]
and so taking the limit of \(\epsilon\) to zero and squaring both sides,
\begin{equation}\label{eqn:sigma_lim}
	\lim_{\epsilon\downarrow 0}\norm{\frac{1}{\epsilon^2}\var{y_t^{(\epsilon)}}} =
	\lim_{\epsilon\downarrow 0}\norm{\frac{1}{\epsilon^2}\var{l_t^{(\epsilon)}}} .
\end{equation}
Now, for \(\epsilon > 0\), define
\begin{align*}
	S^2_{(\epsilon)}(x_0,t) & \coloneqq \sup{\setc{\var{\frac{1}{\epsilon} p^{\T}\left(y_t^{(\epsilon)} - F_0^t\!\left(x_0\right)\right)}}{p \in \R^n, \, \norm{p} = 1}} \\
	                        & = \frac{1}{\epsilon^2}\sup{\setc{p^{\T}\var{y_t^{(\epsilon)}}p}{p \in \R^n, \, \norm{p} = 1}}
\end{align*}
Since \(\var{y_t^{(\epsilon)}}\) is symmetric and positive definite, the Cholesky decomposition provides a lower triangular \(n \times n\) matrix \(\Pi^{(\epsilon)}\) such that \(\var{y_t^{(\epsilon)}} = \Pi^{(\epsilon)}\left[\Pi^{(\epsilon)}\right]^{\T}\), allowing us to write
\begin{align*}
	S^2_{(\epsilon)}(x_0,t) & = \frac{1}{\epsilon^2}\sup\setc{\norm{\Pi^{(\epsilon)}p}^2}{p \in \R^n, \, \norm{p} = 1} \\
	                        & = \frac{1}{\epsilon^2}\norm{\Pi^{(\epsilon)}}^2                                          \\
	                        & = \frac{1}{\epsilon^2}\norm{\var{y_t^{(\epsilon)}(x)}},
\end{align*}
using properties of the spectral norm.
Taking the limit as \(\epsilon\) approaches zero and using \cref{eqn:sigma_lim},
\[
	S^2(x_0,t) = \lim_{\epsilon\downarrow 0} S^2_{(\epsilon)}(x_0,t) = \lim_{\epsilon\downarrow 0}\norm{\frac{1}{\epsilon^2}\var{y_t^{(\epsilon)}}} = \norm{\Sigma_0^t\!\left(x_0\right)},
\]
where \( \Sigma_0^t \) is defined in \Cref{eqn:sigma_def}.
Since \(\Sigma_0^t\!\left(x_0\right)\) is symmetric and positive definite, the operator norm, and therefore \(S^2\!\left(x_0,t\right)\), is given by the largest eigenvalue of \(\Sigma_0^t\!\left(x_0\right)\).

\renewcommand{\bibsection}{}
\par\addvspace{.25in}
\begin{center}
	\footnotesize\uppercase\expandafter{References}
\end{center}

\bibliographystyle{abbrvnat}
{\footnotesize\bibliography{sde_refs}}

\begin{thebibliography}{39}
\providecommand{\natexlab}[1]{#1}
\providecommand{\url}[1]{\texttt{#1}}
\expandafter\ifx\csname urlstyle\endcsname\relax
  \providecommand{\doi}[1]{doi: #1}\else
  \providecommand{\doi}{doi: \begingroup \urlstyle{rm}\Url}\fi

\bibitem[Applebaum(2004)]{Applebaum_2004_LevyProcessesStochastic}
D.~Applebaum.
\newblock \emph{L\'evy {{Processes}} and {{Stochastic Calculus}}}.
\newblock Cambridge {{Studies}} in {{Advanced Mathematics}}. {Cambridge University Press}, {Cambridge, United Kingdom}, 1st edition, 2004.
\newblock ISBN 978-0-511-21119-5.

\bibitem[Archambeau et~al.(2007)Archambeau, Cornford, Opper, and {Shawe-Taylor}]{ArchambeauEtAl_2007_GaussianProcessApproximations}
C.~Archambeau, D.~Cornford, M.~Opper, and J.~{Shawe-Taylor}.
\newblock Gaussian {{Process Approximations}} of {{Stochastic Differential Equations}}.
\newblock In \emph{Gaussian {{Processes}} in {{Practice}}}, pages 1--16. {PMLR}, Mar. 2007.

\bibitem[Badza et~al.(2023)Badza, Mattner, and Balasuriya]{BadzaEtAl_2023_HowSensitiveAre}
A.~Badza, T.~W. Mattner, and S.~Balasuriya.
\newblock How sensitive are {{Lagrangian}} coherent structures to uncertainties in data?
\newblock \emph{Physica D: Nonlinear Phenomena}, 444:\penalty0 133580, Feb. 2023.
\newblock \doi{10.1016/j.physd.2022.133580}.

\bibitem[Balasuriya(2020)]{Balasuriya_2020_StochasticSensitivityComputable}
S.~Balasuriya.
\newblock Stochastic {{Sensitivity}}: {{A Computable Lagrangian Uncertainty Measure}} for {{Unsteady Flows}}.
\newblock \emph{SIAM Review}, 62:\penalty0 781--816, Nov. 2020.
\newblock \doi{10.1137/18M1222922}.

\bibitem[Balasuriya et~al.(2018)Balasuriya, Ouellette, and Rypina]{BalasuriyaEtAl_2018_GeneralizedLagrangianCoherent}
S.~Balasuriya, N.~T. Ouellette, and I.~I. Rypina.
\newblock Generalized {{Lagrangian}} coherent structures.
\newblock \emph{Physica D: Nonlinear Phenomena}, 372:\penalty0 31--51, June 2018.
\newblock \doi{10.1016/j.physd.2018.01.011}.

\bibitem[{Balibrea-Iniesta} et~al.(2016){Balibrea-Iniesta}, Lopesino, Wiggins, and Mancho]{Balibrea-IniestaEtAl_2016_LagrangianDescriptorsStochastic}
F.~{Balibrea-Iniesta}, C.~Lopesino, S.~Wiggins, and A.~M. Mancho.
\newblock Lagrangian {{Descriptors}} for {{Stochastic Differential Equations}}: {{A Tool}} for {{Revealing}} the {{Phase Portrait}} of {{Stochastic Dynamical Systems}}.
\newblock \emph{International Journal of Bifurcation and Chaos}, 26\penalty0 (13):\penalty0 1630036, Dec. 2016.
\newblock \doi{10.1142/S0218127416300366}.

\bibitem[Berner et~al.(2017)Berner, Achatz, Batt{\'e}, Bengtsson, de~la C{\'a}mara, Christensen, Colangeli, Coleman, Crommelin, Dolaptchiev, Franzke, Friederichs, Imkeller, J{\"a}rvinen, Juricke, Kitsios, Lott, Lucarini, Mahajan, Palmer, Penland, Sakradzija, von Storch, Weisheimer, Weniger, Williams, and Yano]{BernerEtAl_2017_StochasticParameterizationNew}
J.~Berner, U.~Achatz, L.~Batt{\'e}, L.~Bengtsson, A.~de~la C{\'a}mara, H.~M. Christensen, M.~Colangeli, D.~R.~B. Coleman, D.~Crommelin, S.~I. Dolaptchiev, C.~L.~E. Franzke, P.~Friederichs, P.~Imkeller, H.~J{\"a}rvinen, S.~Juricke, V.~Kitsios, F.~Lott, V.~Lucarini, S.~Mahajan, T.~N. Palmer, C.~Penland, M.~Sakradzija, J.-S. von Storch, A.~Weisheimer, M.~Weniger, P.~D. Williams, and J.-I. Yano.
\newblock Stochastic {{Parameterization}}: {{Toward}} a {{New View}} of {{Weather}} and {{Climate Models}}.
\newblock \emph{Bulletin of the American Meteorological Society}, 98\penalty0 (3):\penalty0 565--588, Mar. 2017.
\newblock \doi{10.1175/BAMS-D-15-00268.1}.

\bibitem[Bezanson et~al.(2017)Bezanson, Edelman, Karpinski, and Shah]{BezansonEtAl_2017_JuliaFreshApproach}
J.~Bezanson, A.~Edelman, S.~Karpinski, and V.~B. Shah.
\newblock Julia: {{A Fresh Approach}} to {{Numerical Computing}}.
\newblock \emph{SIAM Review}, 59\penalty0 (1):\penalty0 65--98, Jan. 2017.
\newblock \doi{10.1137/141000671}.

\bibitem[Blagoveshchenskii(1962)]{Blagoveshchenskii_1962_DiffusionProcessesDepending}
Y.~N. Blagoveshchenskii.
\newblock Diffusion {{Processes Depending}} on a {{Small Parameter}}.
\newblock \emph{Theory of Probability and its Applications}, 7\penalty0 (2):\penalty0 17, 1962.
\newblock \doi{10.1137/1107013}.

\bibitem[Branicki and Uda(2021)]{BranickiUda_2021_LagrangianUncertaintyQuantification}
M.~Branicki and K.~Uda.
\newblock Lagrangian {{Uncertainty Quantification}} and {{Information Inequalities}} for {{Stochastic Flows}}.
\newblock \emph{SIAM/ASA Journal on Uncertainty Quantification}, 9\penalty0 (3):\penalty0 1242--1313, 2021.

\bibitem[Branicki and Uda(2023)]{BranickiUda_2023_PathBasedDivergenceRates}
M.~Branicki and K.~Uda.
\newblock Path-{{Based Divergence Rates}} and {{Lagrangian Uncertainty}} in {{Stochastic Flows}}.
\newblock \emph{SIAM Journal on Applied Dynamical Systems}, pages 419--482, Mar. 2023.
\newblock \doi{10.1137/21M1466530}.

\bibitem[Br{\'e}maud(2020)]{Bremaud_2020_ProbabilityTheoryStochastic}
P.~Br{\'e}maud.
\newblock \emph{Probability {{Theory}} and {{Stochastic Processes}}}.
\newblock Universitext. {Springer}, 1 edition, 2020.
\newblock ISBN 978-3-030-40182-5.

\bibitem[Budhiraja et~al.(2019)Budhiraja, Friedlander, Guider, Jones, and Maclean]{BudhirajaEtAl_2019_AssimilatingDataModels}
A.~Budhiraja, E.~Friedlander, C.~Guider, C.~K. R.~T. Jones, and J.~Maclean.
\newblock Assimilating {{Data}} into {{Models}}.
\newblock In \emph{Handbook of {{Environmental}} and {{Ecological Statistics}}}. {Chapman and Hall/CRC}, 1st edition edition, 2019.
\newblock ISBN 978-1-315-15250-9.

\bibitem[Freidlin and Wentzell(1998)]{FreidlinWentzell_1998_RandomPerturbationsDynamical}
M.~I. Freidlin and A.~D. Wentzell.
\newblock \emph{Random {{Perturbations}} of {{Dynamical Systems}}}.
\newblock Grundlehren Der Mathematischen {{Wissenschaften}}. {Springer}, {New York, NY}, 2 edition, 1998.

\bibitem[Hadjighasem et~al.(2017)Hadjighasem, Farazmand, Blazevski, Froyland, and Haller]{HadjighasemEtAl_2017_CriticalComparisonLagrangian}
A.~Hadjighasem, M.~Farazmand, D.~Blazevski, G.~Froyland, and G.~Haller.
\newblock A critical comparison of {{Lagrangian}} methods for coherent structure detection.
\newblock \emph{Chaos: An Interdisciplinary Journal of Nonlinear Science}, 27\penalty0 (5):\penalty0 053104, May 2017.
\newblock \doi{10.1063/1.4982720}.

\bibitem[Harlim(2017)]{Harlim_2017_ModelErrorData}
J.~Harlim.
\newblock Model {{Error}} in {{Data Assimilation}}.
\newblock In C.~L.~E. Franzke and T.~J. O'Kane, editors, \emph{Nonlinear and {{Stochastic Climate Dynamics}}}, pages 276--317. {Cambridge University Press}, {Cambridge}, 2017.
\newblock ISBN 978-1-107-11814-0.
\newblock \doi{10.1017/9781316339251.011}.

\bibitem[Hubbard and Hubbard(2009)]{HubbardHubbard_2009_VectorCalculusLinear}
J.~H. Hubbard and B.~B. Hubbard.
\newblock \emph{Vector Calculus, Linear Algebra, and Differential Forms: A Unified Approach}.
\newblock {Matrix Editions}, {Ithaca, NY}, 4th ed edition, 2009.
\newblock ISBN 978-0-9715766-5-0.

\bibitem[Jazwinski(2014)]{Jazwinski_2014_StochasticProcessesFiltering}
A.~H. Jazwinski.
\newblock \emph{Stochastic {{Processes}} and {{Filtering Theory}}}, volume~64 of \emph{Mathematics in {{Science}} and {{Engineering}}}.
\newblock {Elsevier Science}, {Burlington}, 2014.
\newblock ISBN 978-0-08-096090-6.

\bibitem[Jolliffe(2002)]{Jolliffe_2002_PrincipalComponentAnalysis}
I.~T. Jolliffe.
\newblock \emph{Principal {{Component Analysis}}}.
\newblock Springer {{Series}} in {{Statistics}}. {Springer}, {New York, NY}, 2 edition, 2002.
\newblock ISBN 978-0-387-22440-4.

\bibitem[Kallianpur and Sundar(2014)]{KallianpurSundar_2014_StochasticAnalysisDiffusion}
G.~Kallianpur and P.~Sundar.
\newblock \emph{Stochastic Analysis and Diffusion Processes}.
\newblock Number~24 in Oxford Graduate Texts in Mathematics. {Oxford University Press}, {Oxford, United Kingdom}, first edition edition, 2014.
\newblock ISBN 978-0-19-965706-3 978-0-19-965707-0.

\bibitem[Kamenkovich et~al.(2015)Kamenkovich, Rypina, and Berloff]{KamenkovichEtAl_2015_PropertiesOriginsAnisotropic}
I.~Kamenkovich, I.~I. Rypina, and P.~Berloff.
\newblock Properties and {{Origins}} of the {{Anisotropic Eddy-Induced Transport}} in the {{North Atlantic}}.
\newblock \emph{Journal of Physical Oceanography}, 45\penalty0 (3):\penalty0 778--791, Mar. 2015.
\newblock \doi{10.1175/JPO-D-14-0164.1}.

\bibitem[Kasz{\'a}s and Haller(2020)]{KaszasHaller_2020_UniversalUpperEstimate}
B.~Kasz{\'a}s and G.~Haller.
\newblock Universal upper estimate for prediction errors under moderate model uncertainty.
\newblock \emph{Chaos: An Interdisciplinary Journal of Nonlinear Science}, 30\penalty0 (11):\penalty0 113144, Nov. 2020.
\newblock \doi{10.1063/5.0021665}.

\bibitem[Law et~al.(2015)Law, Stuart, and Zygalakis]{LawEtAl_2015_DataAssimilationMathematical}
K.~Law, A.~Stuart, and K.~Zygalakis.
\newblock \emph{Data {{Assimilation}}: {{A Mathematical Introduction}}}.
\newblock Number volume 62 in Texts in Applied Mathematics. {Springer}, {Cham Heidelberg New York Dordrecht London}, 2015.
\newblock ISBN 978-3-319-20324-9 978-3-319-36687-6.
\newblock \doi{10.1007/978-3-319-20325-6}.

\bibitem[Leutbecher et~al.(2017)Leutbecher, Lock, Ollinaho, Lang, Balsamo, Bechtold, Bonavita, Christensen, Diamantakis, Dutra, English, Fisher, Forbes, Goddard, Haiden, Hogan, Juricke, Lawrence, MacLeod, Magnusson, Malardel, Massart, Sandu, Smolarkiewicz, Subramanian, Vitart, Wedi, and Weisheimer]{LeutbecherEtAl_2017_StochasticRepresentationsModel}
M.~Leutbecher, S.-J. Lock, P.~Ollinaho, S.~T.~K. Lang, G.~Balsamo, P.~Bechtold, M.~Bonavita, H.~M. Christensen, M.~Diamantakis, E.~Dutra, S.~English, M.~Fisher, R.~M. Forbes, J.~Goddard, T.~Haiden, R.~J. Hogan, S.~Juricke, H.~Lawrence, D.~MacLeod, L.~Magnusson, S.~Malardel, S.~Massart, I.~Sandu, P.~K. Smolarkiewicz, A.~Subramanian, F.~Vitart, N.~Wedi, and A.~Weisheimer.
\newblock Stochastic representations of model uncertainties at {{ECMWF}}: State of the art and future vision.
\newblock \emph{Quarterly Journal of the Royal Meteorological Society}, 143\penalty0 (707):\penalty0 2315--2339, 2017.
\newblock \doi{10.1002/qj.3094}.

\bibitem[Maclean et~al.(2017)Maclean, Santitissadeekorn, and Jones]{MacleanEtAl_2017_CoherentStructureApproach}
J.~Maclean, N.~Santitissadeekorn, and C.~K. R.~T. Jones.
\newblock A coherent structure approach for parameter estimation in {{Lagrangian Data Assimilation}}.
\newblock \emph{Physica D: Nonlinear Phenomena}, 360:\penalty0 36--45, Dec. 2017.
\newblock \doi{10.1016/j.physd.2017.08.007}.

\bibitem[Mazzoni(2008)]{Mazzoni_2008_ComputationalAspectsContinuous}
T.~Mazzoni.
\newblock Computational aspects of continuous\textendash discrete extended {{Kalman-filtering}}.
\newblock \emph{Computational Statistics}, 23\penalty0 (4):\penalty0 519--539, Oct. 2008.
\newblock \doi{10.1007/s00180-007-0094-4}.

\bibitem[Morzfeld et~al.(2018)Morzfeld, Adams, Lunderman, and Orozco]{MorzfeldEtAl_2018_FeaturebasedDataAssimilation}
M.~Morzfeld, J.~Adams, S.~Lunderman, and R.~Orozco.
\newblock Feature-based data assimilation in geophysics.
\newblock \emph{Nonlinear Processes in Geophysics}, 25\penalty0 (2):\penalty0 355--374, May 2018.
\newblock \doi{10.5194/npg-25-355-2018}.

\bibitem[{\O}ksendal(2003)]{Oksendal_2003_StochasticDifferentialEquations}
B.~{\O}ksendal.
\newblock \emph{Stochastic {{Differential Equations}}}.
\newblock Universitext. {Springer}, {Berlin, Heidelberg}, 6th edition, 2003.
\newblock ISBN 978-3-540-04758-2 978-3-642-14394-6.
\newblock \doi{10.1007/978-3-642-14394-6}.

\bibitem[Palmer(2019)]{Palmer_2019_StochasticWeatherClimate}
T.~N. Palmer.
\newblock Stochastic weather and climate models.
\newblock \emph{Nature Reviews Physics}, 1\penalty0 (7):\penalty0 463--471, July 2019.
\newblock \doi{10.1038/s42254-019-0062-2}.

\bibitem[Pierrehumbert(1991)]{Pierrehumbert_1991_ChaoticMixingTracer}
R.~T. Pierrehumbert.
\newblock Chaotic mixing of tracer and vorticity by modulated travelling {{Rossby}} waves.
\newblock \emph{Geophysical \& Astrophysical Fluid Dynamics}, 58\penalty0 (1-4):\penalty0 285--319, July 1991.
\newblock \doi{10.1080/03091929108227343}.

\bibitem[Rackauckas and Nie(2017{\natexlab{a}})]{RackauckasNie_2017_AdaptiveMethodsStochastic}
C.~Rackauckas and Q.~Nie.
\newblock Adaptive methods for stochastic differential equations via natural embeddings and rejection sampling with memory.
\newblock \emph{Discrete and continuous dynamical systems. Series B}, 22\penalty0 (7):\penalty0 2731--2761, 2017{\natexlab{a}}.
\newblock \doi{10.3934/dcdsb.2017133}.

\bibitem[Rackauckas and Nie(2017{\natexlab{b}})]{RackauckasNie_2017_DifferentialEquationsJlPerformant}
C.~Rackauckas and Q.~Nie.
\newblock {{DifferentialEquations}}.jl \textendash{} {{A Performant}} and {{Feature-Rich Ecosystem}} for {{Solving Differential Equations}} in {{Julia}}.
\newblock \emph{Journal of Open Research Software}, 5\penalty0 (1):\penalty0 15, May 2017{\natexlab{b}}.
\newblock \doi{10.5334/jors.151}.

\bibitem[Reich and Cotter(2015)]{ReichCotter_2015_ProbabilisticForecastingBayesian}
S.~Reich and C.~Cotter.
\newblock \emph{Probabilistic {{Forecasting}} and {{Bayesian Data Assimilation}}}.
\newblock {Cambridge University Press}, {Cambridge}, 2015.
\newblock ISBN 978-1-107-06939-8.
\newblock \doi{10.1017/CBO9781107706804}.

\bibitem[R{\"o}{\ss}ler(2010)]{Rossler_2010_RungeKuttaMethodsStrong}
A.~R{\"o}{\ss}ler.
\newblock Runge-{{Kutta Methods}} for the {{Strong Approximation}} of {{Solutions}} of {{Stochastic Differential Equations}}.
\newblock \emph{SIAM Journal on Numerical Analysis}, 48\penalty0 (3):\penalty0 922--952, 2010.
\newblock \doi{10.1137/09076636X}.

\bibitem[Samelson and Wiggins(2006)]{SamelsonWiggins_2006_LagrangianTransportGeophysical}
R.~M. Samelson and S.~Wiggins.
\newblock \emph{Lagrangian {{Transport}} in {{Geophysical Jets}} and {{Waves}}: {{The Dynamical Systems Approach}}}, volume~31 of \emph{Interdisciplinary {{Applied Mathematics}}}.
\newblock {Springer}, {New York, NY}, 2006.

\bibitem[{Sanz-Alonso} and Stuart(2017)]{Sanz-AlonsoStuart_2017_GaussianApproximationsSmall}
D.~{Sanz-Alonso} and A.~M. Stuart.
\newblock Gaussian {{Approximations}} of {{Small Noise Diffusions}} in {{Kullback-Leibler Divergence}}.
\newblock \emph{Communications in Mathematical Sciences}, 15\penalty0 (7):\penalty0 2087--2097, Oct. 2017.
\newblock \doi{https://dx.doi.org/10.4310/CMS.2017.v15.n7.a13}.

\bibitem[S{\"a}rkk{\"a} and Solin(2019)]{SarkkaSolin_2019_AppliedStochasticDifferential}
S.~S{\"a}rkk{\"a} and A.~Solin.
\newblock \emph{Applied {{Stochastic Differential Equations}}}.
\newblock Institute of {{Mathematical Statistics Textbooks}}. {Cambridge University Press}, {Cambridge}, 2019.
\newblock ISBN 978-1-316-51008-7.
\newblock \doi{10.1017/9781108186735}.

\bibitem[{Schlueter-Kuck} and Dabiri(2019)]{Schlueter-KuckDabiri_2019_ModelParameterEstimation}
K.~L. {Schlueter-Kuck} and J.~O. Dabiri.
\newblock Model parameter estimation using coherent structure colouring.
\newblock \emph{Journal of Fluid Mechanics}, 861:\penalty0 886--900, Feb. 2019.
\newblock \doi{10.1017/jfm.2018.898}.

\bibitem[Sura et~al.(2005)Sura, Newman, Penland, and Sardeshmukh]{SuraEtAl_2005_MultiplicativeNoiseNonGaussianity}
P.~Sura, M.~Newman, C.~Penland, and P.~Sardeshmukh.
\newblock Multiplicative {{Noise}} and {{Non-Gaussianity}}: {{A Paradigm}} for {{Atmospheric Regimes}}?
\newblock \emph{Journal of the Atmospheric Sciences}, 62\penalty0 (5):\penalty0 1391--1409, May 2005.
\newblock \doi{10.1175/JAS3408.1}.

\end{thebibliography}

\end{document}